\documentclass[amstex,12pt,reqno]{amsart}
\usepackage{amsmath, amssymb, amsthm}
\usepackage{mathrsfs}
\usepackage[all]{xy}
\usepackage{color}

\newtheorem{theorem}{Theorem}[section]
\newtheorem{lemma}[theorem]{Lemma}
\newtheorem{proposition}[theorem]{Proposition}
\newtheorem{corollary}[theorem]{Corollary}

\theoremstyle{definition}

\newtheorem*{remark}{Remark}

\title
[The VMO-Teichm\"uller space]{The VMO-Teichm\"uller space and the variant of Beurling--Ahlfors extension by heat kernel}

\author[H. Wei]{Huaying Wei} 
\address{Department of Mathematics and Statistics, Jiangsu Normal University \endgraf Xuzhou 221116, PR China} 
\curraddr{Department of Mathematics, School of Education, Waseda University \endgraf
Shinjuku, Tokyo 169-8050, Japan}
\email{hywei@jsnu.edu.cn} 

\author[K. Matsuzaki]{Katsuhiko Matsuzaki}
\address{Department of Mathematics, School of Education, Waseda University \endgraf
Shinjuku, Tokyo 169-8050, Japan}
\email{matsuzak@waseda.jp}

\makeatletter
\@namedef{subjclassname@2020}{%
\textup{2020} Mathematics Subject Classification}
\makeatother
\thanks{Research supported by 
Japan Society for the Promotion of Science (KAKENHI 18H01125 and 21F20027).}

\begin{document}

\maketitle

\begin{abstract}
We give a real-analytic section for the Teichm\"uller projection onto
the VMO-Teichm\"uller space by using the variant of Beurling--Ahlfors extension by heat kernel introduced by Fefferman, 
Kenig and Pipher in 1991. Based on this result, 
we prove that the VMO-Teichm\"uller space can be endowed with a real Banach manifold structure that is real-analytically equivalent to its complex Banach manifold structure. We also obtain that
the VMO-Teichm\"uller space admits a real-analytic contraction mapping. 
\end{abstract}

\section{Introduction}
\subsection{Background on the universal Teichm\"uller space}
A sense-preserving homeo\-morphism $h$ of the unit circle $\mathbb{S} = \{z\in\mathbb{C} \mid  |z|=1\}$ onto itself
is said to be {\it quasisymmetric} if there exists a least positive constant $C(h)$, called the quasisymmetry constant of $h$, such that 
\[
\frac{\mid h(I_1)\mid}{\mid h(I_2) \mid}\leqslant C(h)
\]
for all pairs of adjacent intervals $I_1$ and $I_2$ on $\mathbb{S}$ with the same length $|I_1|=|I_2|$. 
Beurling and Ahlfors \cite{BA} in 1956 proved that a sense-preserving homeomorphism $h$ of $\mathbb{S}$ is quasi\-symmetric if and only if there exists some quasiconformal homeomorphism of the unit disk $\mathbb{D} = \{z\in\mathbb{C} \mid  |z|<1\}$ onto itself that has boundary value $h$. Later, Douady and Earle \cite{DE} in 1986 gave a quasiconformal extension of a quasisymmetric homeomorphism of $\mathbb{S}$ in a conformally natural way. 

The {\it universal Teich\-m\"ul\-ler space} $T$ is a universal parameter space of marked complex structures on
all Riemann surfaces and can be defined as the group $\mbox{QS}(\mathbb S)$ of all quasisymmetric homeomorphisms of 
$\mathbb{S}$ modulo the left action of the group 
$\mbox{\rm M\"ob}(\mathbb{S})$ of all M\"obius transformations of $\mathbb{S}$, i.e., 
$T= \mbox{\rm M\"ob}(\mathbb{S}) \backslash \mbox{QS}(\mathbb S)$. The quotient by $\mbox{\rm M\"ob}(\mathbb{S})$ is alternatively
achieved by giving a normalization to the elements in $\mbox{QS}(\mathbb S)$. We identify $T$ as the set of all quasisymmetric homeomorphisms of $\mathbb S$ that fix three points, $1$, $-1$ and $i$. This is the real model of the universal Teichm\"uller space $T$. 

Let us recall other descriptions of the universal Teich\-m\"ul\-ler space $T$. 
Let $M(\mathbb D)$ denote the set of all Beltrami coefficients on $\mathbb D$. 
By the measurable Riemann mapping theorem, for any $\mu \in M(\mathbb D)$, 
there is a unique quasiconformal homeomorphism $f_{\mu}$ of $\mathbb D$ onto itself whose complex dilatation is $\mu$ and 
that is normalized by fixing $1$, $-1$ and $i$. There is also a unique quasiconformal homeomorphism 
$f^{\mu}$ of the Riemann sphere $\widehat{\mathbb C}$ with complex dilatation $\mu$ in $\mathbb D$ and 
$0$ in $\mathbb D^* = \hat{\mathbb C} - \overline{\mathbb D}$ that is normalized, for example, by the Laurent expansion
$f^{\mu}(z) = z + \frac{b_1}{z} + \cdots$ at $\infty$. 
The image curve $\Gamma$ of the unit circle $\mathbb S$ under $f^{\mu}$ is called a quasicircle.  
Let $\Omega$ denote the inner domain of $\Gamma$, and let $g$ be the conformal mapping of $\mathbb D$ onto $\Omega$ such that 
$f_{\mu}|_{\mathbb S} = (g|_{\mathbb S})^{-1}\circ f^{\mu}|_{\mathbb S}$. This gives the conformal welding with respect to $f^{\mu}$,
and $f_\mu|_{\mathbb S}$ belongs to $\mbox{QS}(\mathbb S)$.
We say that two elements $\mu$ and $\nu$ in $M(\mathbb D)$ are equivalent, denoted by $\mu \sim \nu$, if $f^{\mu}|_{\mathbb D^*} = f^{\nu}|_{\mathbb D^*}$, or equivalently, $f_{\mu}|_{\mathbb S} = f_{\nu}|_{\mathbb S}$. Therefore, the universal Teichm\"uller space $T$ can also be defined as follows.

\begin{itemize}
\item
$T$ is the set $M(\mathbb D)/\hspace{-2mm}\sim$ of all equivalence classes $[\mu]$. This is the Beltrami coefficient model of $T$ $([\mu]\mapsto f_{\mu}|_{\mathbb S})$.

\item
$T$ is the set of all quasicircles (up to a M\"obius transformation of  $\hat{\mathbb C}$). Here, a Jordan curve $\Gamma$ is said to be a {\it quasicircle} if, by definition, there is a positive constant $C(\Gamma)$ such that
$$
\frac{\mbox{diameter}(\widetilde{z_1z_2})}{|z_1 - z_2|} \leq C(\Gamma)
$$
for the smaller subarc $\widetilde{z_1z_2}$ of $\Gamma$ joining any two finite points $z_1$ and $z_2$ on $\Gamma$. 
This is the Jordan curve model of $T$ $([\mu] \mapsto f^{\mu}(\mathbb S))$. 

\item
$T$ is the set of all conformal mappings $f^{\mu}$ on $\mathbb D^*$ (up to a M\"obius transformation of $\hat{\mathbb C}$)  which can be extended to a quasiconformal mapping in $\hat{\mathbb C}$. This is the conformal mapping model of $T$ 
$([\mu] \mapsto f^{\mu}|_{\mathbb D^*})$.
\end{itemize}

The different models of $T$ are identified with each other by the one-to-one maps given in the brackets.  
It is known that $T$ has a unique complex Banach manifold structure such that the Teichm\"uller projection $\pi: M(\mathbb D) \to T$ sending $\mu$ to the equivalence class $[\mu]$  is a holomorphic split submersion in the exact sense
(see \cite[Theorem V.5.3]{Le}, \cite[Sections 3.4, 3.5]{Na}).

\subsection{Background on the VMO-Teichm\"uller space}
The VMO-Teichm\"uller space $T_v$ is a subspace of the universal Teichm\"uller space $T$ introduced from the viewpoint of harmonic analysis. Its real model is the set of all normalized strongly symmetric homeomorphisms of $\mathbb S$. Here, by  
a {\it strongly symmetric homeomorphism} $h$ of $\mathbb S$, we mean that it is absolutely continuous with $\log h' \in \mbox{VMO}(\mathbb S)$, the function space on $\mathbb S$ of vanishing mean oscillation.  By the results of Pommerenke \cite{Pom78} in 1978, Dyn$'$kin \cite{Dy} in 1997, and Shen and Wei \cite{SW} in 2013, the following conditions are equivalent (see \cite{SW} for the comprehensive information):
\begin{itemize}
\item
[$(V_0)$] 
$f_{\mu}|_{\mathbb S}$ is a strongly symmetric homeomorphism of $\mathbb S$.
\end{itemize}

A positive measure $\lambda$ defined on $\mathbb D$ 
is called a {\it Carleson measure} if
\begin{equation}\label{CM}
\Vert \lambda \Vert_c = \sup \frac{\lambda(S_{h,\theta_0})}{h}
\end{equation}
is finite, where the supremum is taken over all sectors
\begin{align}\label{sector}
S_{h,\theta_0} &= \{re^{i\theta} \in \mathbb D \mid 1-h \leqslant r < 1, \;|\theta - \theta_0| \leqslant \pi h\} 
\end{align}
for $h \in (0, 1]$ and $\theta_0 \in [0, 2\pi)$.
A Carleson measure $\lambda$ is called a {\it vanishing Carleson measure} if 
\begin{equation*}\label{CM0}
\lim_{h \to 0} \frac{\lambda(S_{h,\theta_0})}{h} = 0
\end{equation*}
uniformly for $\theta_0 \in [0, 2\pi)$. We denote by $\mbox{CM}(\mathbb D)$ and $\mbox{CM}_0(\mathbb D)$ the set of all Carleson measures and vanishing Carleson measures on $\mathbb D$, respectively. Such measures on $\mathbb D^*$ are defined similarly.

\begin{itemize}
\item[($V_1$)] 
There exists some $f_\mu$ whose complex dilatation induces a vanishing Carleson measure 
$\lambda_{\mu} = |\mu(z)|^2/(1 - |z|^2)dxdy \in \mbox{CM}_0(\mathbb D)$.   
\end{itemize}

Let ${\mathcal M}_0(\mathbb D) \subset M(\mathbb D)$ be the subset of Beltrami coefficients $\mu$
that induce vanishing Carleson measures $\lambda_{\mu}$. 
The Beltrami coefficient model of the VMO-Teichm\"uller space $T_v=T_v(\mathbb D)$ is given by the image of the 
Teichm\"uller projection $\pi:{\mathcal M}_0(\mathbb D) \to T_v$. 

\begin{itemize}
\item[($V_2$)] $\Gamma = f^{\mu}(\mathbb S)$ is a bounded {\it asymptotically smooth curve in the sense of Pommerenke} \cite{Pom78}. Namely, 
$$
\lim_{t \to 0}\sup_{|z_1 - z_2| \leq t} \frac{\mbox{arc-length}(\widetilde{z_1z_2})}{|z_1 - z_2|} = 1, \qquad \mbox{for any } z_1, z_2 \in \Gamma. 
$$

\item[($V_3$)] $\mathcal L_{f^{\mu}} = \log (f^{\mu})'$ belongs to $\mbox{VMOA}(\mathbb D^*)$, the space of analytic functions in $\mathbb D^*$ of vanishing mean oscillation. This can be also characterized by the condition that  
$\mathcal S_{f^{\mu}} = \mathcal L_{f^{\mu}}'' - \frac{1}{2}(\mathcal L_{f^{\mu}}')^2$ belongs to $\mathcal B_0(\mathbb D^*)$.
\end{itemize}

Here, an analytic function $\phi$ on $\mathbb D^*$ belongs to $\mbox{VMOA}(\mathbb D^*)$
if and only if 
\begin{equation}\label{CML}
\lambda_\phi=|\phi'(z)|^2(|z|^2 - 1) dxdy \in \mbox{CM}_0(\mathbb D^*).
\end{equation}
Moreover,
let $\mathcal B_0(\mathbb D^*)$ denote the Banach space of holomorphic functions $\varphi$ in $\mathbb D^*$ each of which induces a vanishing Carleson measure 
\begin{equation}\label{CMS}
\lambda_{\varphi} = |\varphi(z)|^2(|z|^2 - 1)^3 dxdy \in \mbox{CM}_0(\mathbb D^*). 
\end{equation}
Via the Bers embedding given by the Schwarzian derivative $\mathcal S_{f^{\mu}|_{\mathbb D^*}}$, $T_v$ is embedded into 
the space $\mathcal B_0(\mathbb D^*)$, which endows $T_v$ with a natural complex Banach manifold structure, so that the Teichm\"uller projection $\pi$ from $\mathcal M_0(\mathbb D)$ to $T_v$ is holomorphic with a local holomorphic right inverse 
(see \cite[Theorem 5.1]{SW}).   

\subsection{Main results and plan of the paper}
In 1974, Earle \cite{E1} proved that the Beurling--Ahlfors extension gives the Teichm\"uller projection $\pi$ 
a global real-analytic section (right inverse) for the universal Teichm\"uller space $T$.  
In 1986, Douady and Earle \cite{DE} proved that this is also true for the conformally barycentric extension.
In 2016, Tang, Wei and Shen \cite{TWS} showed that the conformally barycentric extension
yields a continuous section for the VMO-Teichm\"uller space $T_v$. However, the existence of the real-analytic section for $T_v$ is not known yet in the literature. In this paper, we provide a real-analytic section for $T_v$. 
Our real-analytic section is given by the variant of Beurling--Ahlfors extension by heat kernel introduced by Fefferman, Kenig and Pipher \cite{FKP} in 1991 (see Section 4). 

We regard $T_v$ as its real model. For any $h \in T_v$ with $u=\log|h'|$, 
we denote $\mu_u$ the complex dilatation of the variant of Beurling--Ahlfors extension by heat kernel of $h$. 
Precisely speaking, this quasiconformal extension is done from the real line to the upper half-plane,
but under the lift and the projection by the universal cover $z \mapsto e^{2\pi iz}$,
we can also consider this extension from the unit circle to the unit disk.
Let ${\rm VMO}_{\mathbb R}(\mathbb S)$ be the real Banach space of
all real-valued functions on $\mathbb S$ of vanishing mean oscillation.  
Then, our main results can be summarized as follows. They 
are given in the corresponding claims specified after the statements.
 
\begin{theorem}\label{mains}
The following two statements hold:
\begin{enumerate} 
\item[(i)] The map $\widetilde L: T_v \to {\rm VMO}_{\mathbb R}(\mathbb S)$ 
sending $h$ to $u = \log |h'|$ is real-analytic {\rm (Corollary \ref{holo2}).}
\item[(ii)] The map $\widetilde\Lambda: {\rm VMO}_{\mathbb R}(\mathbb S) \to \mathcal M_0(\mathbb D)$ sending $u$ 
to $\mu_u$ is real-analytic {\rm (Theorem \ref{realana}).}
\end{enumerate}
Consequently, the composition $\widetilde\Lambda\circ \widetilde L: T_v \to \mathcal M_0(\mathbb D)$ is 
real-analytic such that $\pi\circ \widetilde\Lambda\circ \widetilde L$ is the identity on $T_v$. This in particular gives a global real-analytic section
for the Teichm\"uller projection $\pi: \mathcal M_0(\mathbb D) \to T_v$. 
\end{theorem}

Based on the observation that $\widetilde L^{-1}  = \pi\circ \widetilde\Lambda$, we conclude  
that $\widetilde L^{-1}: {\rm VMO}_{\mathbb R}(\mathbb S) \to T_v$ is also real-analytic.    

\begin{corollary}\label{mainss}
$\widetilde L: T_v \to {\rm VMO}_{\mathbb R}(\mathbb S)$ is a real-analytic homeomorphism from $T_v$ onto ${\rm VMO}_{\mathbb R}(\mathbb S)$ whose inverse $\widetilde L^{-1}$ is also real-analytic. 
\end{corollary}

This implies that the real model of $T_v$ can be endowed with a real Banach manifold structure from ${\rm VMO}_{\mathbb R}(\mathbb S)$ by $\widetilde L$, which is real-analytically equivalent to the 
natural complex Banach manifold structure on $T_v$ defined by ($V_3$). 
 
The second immediate consequence of Theorem \ref{mains} is the contractibility of $T_v$
in a stronger sense:

\begin{corollary}
The VMO-Teichm\"uller space $T_v$ admits
a real-analytic contraction. Precisely, the contraction $\Phi:T_v\times [0, 1] \to T_v$ is defined by 
$\Phi(h, t)= \pi\left((1 - t)\widetilde\Lambda\circ \widetilde L (h)\right)$, and $\Phi$ is 
a real-analytic mapping. Alternatively, another real-analytic contraction $\Phi'$ can be given by
$\Phi'(h,t)=\widetilde L^{-1}((1-t) \widetilde L(h))$.
\end{corollary}

Fan and Hu \cite{FH} proved that the image domain of $T_v$ under the Bers embedding is holomorphically contractible 
in the sense that there is a contraction $\Phi:T_v\times [0, 1] \to T_v$ such that $\Phi$ is continuous and
$\Phi(\cdot, t)$ is holomorphic for each fixed $t$.

We remark that the corresponding claim to Theorem \ref{mains} for the BMO Teichm\"uller space $T_b=\pi(\mathcal M(\mathbb D))$
is open, where $\mathcal M(\mathbb D)$ is the set of Beltrami coefficients $\mu$ on $\mathbb D$ such that 
$\lambda_\mu \in {\rm CM}(\mathbb D)$. The crucial difference stems from the fact that
$L_{\mathbb R}^\infty(\mathbb R)$ is dense in ${\rm VMO}_{\mathbb R}(\mathbb S)$ in the BMO-norm (see Section 6).
Because of this property, the estimates performed in Section 5 will be available.

We end this introduction section with noticing the organization of this paper. 
In Section 2, we recall definitions and properties of the BMO and VMO spaces, Muckenhoupt $A_{\infty}$-weights, and strongly quasisymmetric homeomorphisms, which will be frequently used in the rest of the paper. In Section 3, we prove the first statement of Theorem \ref{mains} (i.e. Corollary \ref{holo2}). 
In Section 4, we review the variant of Beurling--Ahlfors extension by heat kernel, and in Section 5,
we show the local boundedness of the complex dilatations made by such extensions of strongly symmetric homeomorphisms.
Section 6 is devoted to the proof of the second statement of Theorem \ref{mains} (i.e. Theorem \ref{realana}).

\section{Preliminaries on BMO, VMO, and $A_\infty$-weights}
We treat the cases of $\mathbb R$ and $\mathbb S$ simultaneously;
$X$ denotes either of $\mathbb R$ or $\mathbb S$. 
Let $I$ be an interval in $X$. 
A locally integrable complex-valued function $u$ on $I$ is of {\it BMO} if
$$
\Vert u \Vert_{{\rm BMO}(I)}=\sup_{J \subset I}\frac{1}{|J|} \int_J |u(x)-u_J| |dx| <\infty,
$$
where the supremum is taken over all bounded intervals $J$ on $I$ and $u_J$ denotes the integral mean of $u$
over $J$. The set of all BMO functions on $I$ is denoted by ${\rm BMO}(I)$.  This is regarded as a Banach space with norm 
$\Vert \cdot \Vert_{{\rm BMO}(I)}$ modulo constants.  
Hereafter, $\Vert \cdot \Vert_{{\rm BMO}(X)}$ will be simplified to be $\Vert \cdot \Vert_{*}$. 
It is said that $u \in {\rm BMO}(I)$ is of {\it VMO} if
$$ 
\lim_{|J| \to 0}\frac{1}{|J|} \int_J |u(x)-u_J| |dx|=0,
$$
and the set of all such functions is denoted by ${\rm VMO}(I)$.
This is a closed subspace of ${\rm BMO}(I)$.  The real subspaces of ${\rm BMO}(I)$ and ${\rm VMO}(I)$ consisting of all real-valued BMO functions and VMO functions on  $I$ are denoted by ${\rm BMO}_{\mathbb R}(I)$ and ${\rm VMO}_{\mathbb R}(I)$, respectively. The {\it John--Nirenberg inequality} for BMO functions (see \cite[VI.2]{Ga}, \cite[IV.1.3]{St2}) asserts that
there exists two universal positive constants $C_0$ and $C_{JN}$ such that for any complex-valued BMO function $u$, 
any bounded interval $J$ of $I$, and any $\lambda > 0$, it holds that
\begin{equation}\label{JN}
\frac{1}{|J|} |\{t \in J: |u(t) - u_J| \geq \lambda \}| \leq C_0 \exp\left(\frac{-C_{JN}\lambda}{\Vert u \Vert_{{\rm BMO}(I)}} \right).
\end{equation}

A locally integrable non-negative measurable function $\omega \geq 0$ on $X$ 
is called a {\it weight}. We say that $\omega$ is
a {\it Muckenhoupt  $A_\infty$-weight} (or $A_{\infty}$-weight for simplicity)
if
there are positive constants $\alpha(\omega)$, $K(\omega)>0$ such that  
\begin{equation}\label{SD}
\frac{\int_E \omega(x)|dx|}{\int_I \omega(x)|dx|}\leq K(\omega)\left(\frac{|E|}{|I|}\right)^{\alpha(\omega)}
\end{equation}
for any bounded interval $I \subset X$ and 
for any measurable subset $E \subset I$ (see \cite[Theorem V]{CF}
and \cite[Lemma VI.6.11]{Ga}).
From this condition, we see that an $A_\infty$-weight $\omega$ satisfies the doubling property:
there exists a constant $\rho(\omega)>1$ such that
\begin{equation}\label{doubling}
\int_{2I} \omega(x)|dx| \leq \rho(\omega) \int_{I} \omega(x)|dx|
\end{equation}
holds for any bounded intervals $I$ and $2I$ with the same center and $|2I|=2|I|$.
The doubling constant $\rho(\omega)$ can be given in terms of $\alpha(\omega)$ and $K(\omega)$.

The Jensen inequality implies that
\begin{equation}\label{Jensen}
\exp \left(\frac{1}{|I|} \int_I \log \omega(x) |dx| \right) 
\leq \frac{1}{|I|} \int_I \omega(x) |dx|.
\end{equation}
Another characterization of $A_\infty$-weights can be given by 
the inverse Jensen inequality. 
Namely,
$\omega \geq 0$
belongs to the class of $A_\infty$-weights 
if and only if there exists a constant $C_\infty(\omega) \geq 1$ such that
\begin{equation}\label{iff}
\frac{1}{|I|} \int_I \omega(x) |dx| \leq C_\infty(\omega) \exp \left(\frac{1}{|I|} \int_I \log \omega(x) |dx| \right) 
\end{equation}
for every bounded interval $I \subset X$ (see \cite{Hr}). 

We call the least possible value of such $C_\infty(\omega)$
the $A_\infty$-constant for $\omega$. 
If $\omega$ is an $A_\infty$-weight,
the constants $\alpha(\omega)$ and $K(\omega)$ in (\ref{SD}) are estimated by $C_\infty(\omega)$ 
as is shown in \cite[Theorem 1]{Hr},
and conversely,
$C_\infty(\omega)$ is estimated by $\alpha(\omega)$ and $K(\omega)$
(see \cite[Section 3]{CF}).
One can also refer to \cite[p.218]{St2} for these implications.

By (\ref{iff}) and the Jensen inequality, we see that if $\omega$ is an $A_\infty$-weight on $X$, then
$\log \omega$ belongs to ${\rm BMO}_{\mathbb R}(X)$, and  
conversely, 
if $\Vert \log \omega \Vert_*$
is sufficiently small, then
$\omega$ is an $A_\infty$-weight (see \cite[Lemma VI.6.5]{Ga}, \cite[p.409]{GR}).

A sense-preserving homeomorphism $h$ of $X$ is called {\it strongly quasisymmetric} 
if $h$ is locally absolutely continuous such that $|h'|$ is an $A_\infty$-weight.
In particular, $\log h' \in {\rm BMO}(X)$.    
A strongly quasisymmetric homeomorphism is quasisymmetric. We see from (\ref{SD}) that the
set of all strongly quasisymmetric homeomorphisms of $X$ forms a group. 
Moreover, Jones \cite{Jo} in 1983 characterized a strongly quasisymmetric homeomorphism by the pull-back operator defined on 
the BMO space (see \cite[Lemma]{ABL88} for the case of the VMO space): 

\begin{proposition}\label{pullback}
A sense-preserving homeomorphism $h$ of $X$ is strongly quasisymmetric if and only if 
the pull-back operator $P_h: u \mapsto u\circ h$ gives an isomorphism of ${\rm BMO}(X)$ onto itself,
that is, $P_h$ and $(P_h)^{-1}$ are bounded linear operators. Moreover, in the case of $X=\mathbb S$,
$P_h$ maps ${\rm VMO}(\mathbb S)$ onto ${\rm VMO}(\mathbb S)$ if $h$ is strongly quasisymmetric defined on $\mathbb S$. 
\end{proposition}

It was noticed in our recent paper \cite{WM-5} that the pull-back operator $P_h$ does not necessarily maps ${\rm VMO}(\mathbb R)$ into itself if $h$ is strongly quasisymmetric defined on $\mathbb R$. However, under an extra assumption that
$h$ is uniformly continuous, the operator $P_h$ preserves ${\rm VMO}(\mathbb R)$.

\section{Real-analytic mapping to the space of VMO (Theorem \ref{mains} {\rm (i)})} 

In this section, we will prove the first part of Theorem \ref{mains} whose precise statement will be given in Corollary \ref{holo2}
to Theorem \ref{holo}. 

Let $M(\mathbb D)$ denote the open unit ball of the Banach space $L^{\infty}(\mathbb D)$
of all essentially bounded measurable functions on $\mathbb D$. An element in $M(\mathbb D)$ is called a
Beltrami coefficient.  For $\mu \in L^\infty(\mathbb D)$, we set
$\lambda_\mu = |\mu(z)|^2/(1 - |z|^2) dxdy$ and define
$\Vert \mu \Vert_c =\Vert \lambda_\mu \Vert_c^{1/2}$.
Here, $\Vert \lambda_\mu \Vert_c$ is the Carleson norm of $\lambda_{\mu}$ defined in \eqref{CM}.  
Then, we introduce a new norm 
$\Vert \mu \Vert_{\infty} + \Vert \mu \Vert_{c}$ 
for $\mu$.  
Let $\mathcal{L}(\mathbb D)$ denote the Banach space consisting of all elements $\mu \in L^{\infty}(\mathbb D)$ with norm
$$
\Vert \mu \Vert_{\infty}+\Vert \mu \Vert_{c} < \infty.
$$
Let $\mathcal L_0(\mathbb D)$ denote the subspace of $\mathcal{L}(\mathbb D)$ consisting of all $\mu$
such that $\lambda_\mu$ is a vanishing Carleson measure. 
This is a Banach subspace of $\mathcal{L}(\mathbb D)$. 
Moreover, we define the corresponding spaces of Beltrami coefficients
as $\mathcal{M}(\mathbb D) =  M(\mathbb D) \cap \mathcal{L}(\mathbb D)$ and
$\mathcal{M}_0(\mathbb D) =  M(\mathbb D) \cap \mathcal{L}_0(\mathbb D)$. 
The relevant spaces on $\mathbb D^*$ and $\mathbb U$ can be defined by replacing the hyperbolic density $1/(1 - |z|^2)$ on $\mathbb D$ with that on $\mathbb D^*$ and $\mathbb U$. 

It is well known that a quasiconformal homeomorphism of $\mathbb D$ onto itself induces a biholomorphic automorphism of the universal Teichm\"uller space $T$. Precisely, let $f_{\mu}: \mathbb D \to \mathbb D$ be a quasiconformal homeomorphism 
with complex dilatation $\mu \in M(\mathbb D)$. Then, $f_{\mu}$ induces a biholomorphic automorphism 
$r_{\mu}: M(\mathbb D) \to M(\mathbb D)$ defined as
$$
r_{\mu}(\nu) = \left(\frac{\nu - \mu}{1 - \bar{\mu}\nu} \frac{\partial f_{\mu}}{\overline{\partial f_{\mu}}} \right)\circ f_{\mu}^{-1}, 
$$
which is the complex dilatation of $f_{\nu} \circ f_{\mu}^{-1}$. 
Moreover,
$r_{\mu}$ descends down by the Teichm\"uller projection $\pi$
to a biholomorphic automorphism $R_{[\mu]}: T \to T$ satisfying $R_{[\mu]}\circ\pi = \pi\circ r_{\mu}$  
(see \cite[Section V.5.4]{Le}, \cite[Section 3.6.2]{Na}). 

The corresponding result for $T_v(\mathbb D)=\pi(\mathcal{M}_0(\mathbb D))$ (and for $T_b(\mathbb D)=\pi(\mathcal{M}(\mathbb D))$)
was asserted  in \cite[Remark 5.1]{SW}, which will be used in the proof of Proposition \ref{curve} and Theorem \ref{holo}. 
\begin{proposition}\label{right}
Suppose $f_{\mu}$ is bi-Lipschitz under the hyperbolic metric with complex dilatation $\mu \in \mathcal M_0(\mathbb D)$. Then, 
 $r_{\mu}: \mathcal M_0(\mathbb D) \to \mathcal M_0(\mathbb D)$ is a biholomorphic automorphism of $\mathcal M_0(\mathbb D)$, and moreover $R_{[\mu]}: T_v \to T_v$ is a biholomorphic automorphism of $T_v=T_v(\mathbb D)$. 
\end{proposition}

\begin{remark}
We have that an inequality $\Vert r_{\mu}(\nu)\Vert_c \leq C(\Vert \mu \Vert_{\infty}) \Vert \nu - \mu\Vert_c$ holds, 
which was observed in \cite[Remark 5.1]{SW} by examining the proof of \cite[Lemma 10]{CZ} (see also \cite{WM-5}). This 
implies that the map $r_{\mu}: \mathcal M_0(\mathbb D) \to \mathcal M_0(\mathbb D)$ is locally bounded. 
Following Remark after Theorem \ref{Linfty}, the holomorphy of $r_{\mu}$ can be verified. 
The inverse $r_{\mu}^{-1}$ is treated similarly.
These arguments show that the map $r_{\mu}: \mathcal M_0(\mathbb D) \to \mathcal M_0(\mathbb D)$ is biholomorphic. 
The local inverse of the Bers projection, which is equivalent to the Teichm\"uller projection through the Bers embedding,
constructed in \cite[Theorem 5.1]{SW} is holomorphic for the same reason.
\end{remark}

A sense-preserving homeomorphism $h$ of $\mathbb S$ onto itself  is called {\it strongly symmetric} if $h$ is absolutely continuous with $\log h' \in {\rm VMO}(\mathbb S)$. Then, $h$ must be strongly quasisymmetric since the condition $\log h' \in {\rm VMO}(\mathbb S)$ implies $|h'| \in A_{\infty}$ by the John--Nirenberg inequality \eqref{JN}
(see \cite[p.474]{GR}).  
Let ${\rm SS} (\mathbb S)$ be the set of all strongly symmetric homeomorphisms $h$ of $\mathbb S$ onto itself which satisfy normalized conditions 
\begin{equation}\label{2conditions}
h(1) = 1 \quad {\rm and} \quad \int_{-\pi}^{\pi}h(e^{it})dt = 0. 
\end{equation}
Here, we use these normalized conditions different from  the traditional case of the universal Teichm\"uller space for later consideration.  
By $(V_0)\Rightarrow (V_1)$ in Section 1, any $h \in {\rm SS} (\mathbb S)$ can be extended to a quasiconformal mapping of $\mathbb D$ onto itself such that its complex dilatation belongs to $\mathcal M_0(\mathbb D)$. The Douady--Earle extension $E(h)$ is such an extension, and moreover, it is a bi-Lipschitz diffeomorphism under the hyperbolic metric (see \cite{CZ, TWS}). The second condition in \eqref{2conditions} implies $E(h)(0) = 0$, and $E(h)(\infty) = \infty$ by reflection across $\mathbb S$.

We say that a quasiconformal homeomorphism $G$ of the whole plane $\mathbb C$ 
onto itself is a {\it VMO-quasiconformal mapping relative to $\mathbb S$} if 
$\mu_{G|_{\mathbb D}} \in \mathcal M_0(\mathbb D)$ and $\mu_{G|_{\mathbb D^*}} \in \mathcal M_0(\mathbb D^*)$.  Let $\gamma = G|_{\mathbb S}$ and $\Gamma = \gamma(\mathbb S)$. Then, the homeomorphism $\gamma: \mathbb S \to \Gamma$ is a parametrization of the curve $\Gamma$, and we call $\gamma$ a {\it strongly symmetric embedding}.  For $\mu_1 \in \mathcal M_0(\mathbb D)$ and $\mu_2 \in \mathcal M_0(\mathbb D^*)$, we denote by $G = G(\mu_1, \mu_2)$ the normalized VMO-quasiconformal mapping $G$ relative to
$\mathbb S$ $(G(0) = 0, G(1) = 1, G(\infty) = \infty)$ with $\mu_{G|_{\mathbb D}} = \mu_1$, $\mu_{G|_{\mathbb D^*}} = \mu_2$. 
Let ${\rm SS}_{\mathbb C}(\mathbb S)$ be the set of all normalized strongly symmetric embeddings of $\mathbb S$. 

\begin{proposition}\label{curve}
A normalized strongly symmetric embedding 
$\gamma:\mathbb S \to \mathbb C$ is absolutely continuous and $\log \gamma'$
belongs to ${\rm VMO}(\mathbb S)$. Moreover, the image $\Gamma = \gamma(\mathbb S)$ is a bounded asymptotically smooth curve.  
\end{proposition}

\begin{proof}
We take a normalized VMO-quasiconformal mapping $G(\mu_1, \mu_2)$ relative to
$\mathbb S$
such that $G|_{\mathbb S}=\gamma$. We extend $\mu_2$ to $\mathbb D$ by the reflection
\begin{equation}\label{refl}
\mu_2(z) = \overline{\mu_2\left( \frac{1}{\bar z}\right)}\frac{z^2}{\bar{z}^2},  \qquad z \in \mathbb D, 
\end{equation}
and take the unique normalized quasiconformal homeomorphism $F:\mathbb C \to \mathbb C$ with complex dilatation
$\mu_2$ $(F(0) = 0, F(1) = 1, F(\infty) = \infty)$. By $F(\mathbb S) = \mathbb S$ and
$(V_1) \Rightarrow (V_0)$ in Section 1, we have $f=F|_{\mathbb S}$ is strongly symmetric, 
and in particular, $\log f'$ belongs to ${\rm VMO}(\mathbb S)$. 
Next, we take the unique normalized quasiconformal homeomorphism $H:\mathbb C \to \mathbb C$ 
$(H(0) = 0, H(1) = 1, H(\infty) = \infty)$ that is conformal on $\mathbb D^*$ and
whose complex dilatation on $\mathbb D$ is the push-forward $F_* \mu_1$ of $\mu_1$ by $F$. Namely, 
the complex dilatation of $H \circ F|_{\mathbb D}$ is $\mu_1$. Then,  we have $H \circ F=G$.

The complex dilatation $F_* \mu_1$ is Teichm\"uller equivalent to some $\nu \in \mathcal M_0(\mathbb D)$.
This can be seen from Proposition \ref{right} if we replace $F|_{\mathbb D}$ with the Douady--Earle extension of $f$ which is a bi-Lipschitz diffeomorphism  
under the hyperbolic metric and whose complex dilatation belongs to $\mathcal M_0(\mathbb D)$. 
Then, the curve $\Gamma$ is bounded asymptotically smooth by $(V_1) \Rightarrow (V_2)$ in Section 1. 
Since $\log (H|_{\mathbb D^*})' \in {\rm VMOA}(\mathbb D^*)$ by $(V_1) \Rightarrow (V_3)$,
it has the finite non-tangential limit 
almost everywhere on $\mathbb S$, and this boundary function belongs to ${\rm VMO}(\mathbb S)$
(see \cite[Theorem 9.23]{Zh}).
This also implies that $H|_{\mathbb D^*}$ has a finite angular derivative almost everywhere
on $\mathbb S$ by \cite[Proposition 4.7]{Pom}. Moreover, since $H(\mathbb D^*)$ is a quasidisk,
\cite[Theorem 5.5]{Pom} asserts that the angular derivative at $z \in \mathbb S$ coincides with
$$
h'(z)=\lim_{\mathbb S \ni \xi \to z} \frac{h(\xi)-h(z)}{\xi-z}
$$
for $h=H|_{\mathbb S}$.
This shows that the non-tangential limit of $(H|_{\mathbb D^*})'$ coincides with the ordinary derivative $h'$
almost everywhere on $\mathbb S$. By taking the logarithm, we have that
$\log h' $ belongs to
${\rm VMO}(\mathbb S)$. Hence, $|h'|$ is an $A_{\infty}$-weight, and in particular, $h$ is absolutely continuous on $\mathbb S$.

By $H \circ F=G$, we see that $\gamma=G|_{\mathbb S}$ is also absolutely continuous, 
and taking the derivative on $\mathbb S$,
we have 
$$
\log h' \circ f + \log f'=\log \gamma'.
$$
We have seen that $\log f' \in {\rm VMO}(\mathbb S)$.
Since $f$ is strongly symmetric and $\log h' \in {\rm VMO}(\mathbb S)$,
Proposition \ref{pullback} shows that $\log h' \circ f \in {\rm VMO}(\mathbb S)$.
Thus, we obtain that $\log \gamma' \in {\rm VMO}(\mathbb S)$.
This completes the proof.
\end{proof}

The {\it VMO-Teichm\"uller space} $T_v(\mathbb D)$ on $\mathbb D$ is
defined to be
the set of all Teichm\"uller equivalence classes $[\mu]$ for $\mu \in \mathcal M_0(\mathbb D)$.
The quotient map $\pi:\mathcal M_0(\mathbb D) \to T_v(\mathbb D)$ is called the Teichm\"uller projection.
The VMO-Teichm\"uller space $T_v(\mathbb D^*)$ on  $\mathbb D^*$ 
and related concepts are
defined similarly. 

We define a map
$$
\widetilde \iota:\mathcal M_0(\mathbb D) \times \mathcal M_0(\mathbb D^*) \to {\rm SS}_{\mathbb C}(\mathbb S)
$$
by $\widetilde \iota(\mu_1,\mu_2)=G(\mu_1,\mu_2)|_{\mathbb S}$.
Then, by the argument of simultaneous uniformization due to Bers,
we see the following fact.

\begin{proposition}\label{id}
The space ${\rm SS}_{\mathbb C}(\mathbb S)$  is identified with
$T_v(\mathbb D) \times T_v(\mathbb D^*)$. More precisely, $\widetilde \iota$ splits into a well-defined bijection
$$
\iota:T_v(\mathbb D) \times T_v(\mathbb D^*) \to {\rm SS}_{\mathbb C}(\mathbb S)
$$
by the product of the Teichm\"uller projections
$$
\widetilde \pi:\mathcal M_0(\mathbb D) \times \mathcal M_0(\mathbb D^*) \to T_v(\mathbb D) \times T_v(\mathbb D^*),
$$
such that $\widetilde \iota=\iota \circ \widetilde \pi$.
\end{proposition}

Moreover, via the Bers embedding $\beta$ given by the Schwarzian derivative $\mathcal S_{G(\mu,0)|_{\mathbb D^*}}$, $T_v(\mathbb D)$ is embedded into the complex Banach space $\mathcal B_0(\mathbb D^*)$ with norm defined by
$\Vert \varphi \Vert_{\mathcal B} = \Vert \lambda_{\varphi} \Vert_{c}$ for $\lambda_\varphi$ as in \eqref{CMS}, 
and in the same way $T_v(\mathbb D^*)$ is embedded into the complex Banach space $\mathcal B_0(\mathbb D)$, which is defined similarly to $\mathcal B_0(\mathbb D^*)$:
\begin{align*}
T_v(\mathbb D) &\cong \beta (T_v(\mathbb D)) = \{ {\mathcal S}_{G(\mu,0)} \in \mathcal B_0(\mathbb D^*) \mid \mu \in \mathcal M_0(\mathbb D)\};\\
T_v(\mathbb D^*) &\cong \beta (T_v(\mathbb D^*)) = \{ {\mathcal S}_{G(0,\mu)} \in \mathcal B_0(\mathbb D) \mid \mu \in \mathcal M_0(\mathbb D^*)\}. 
\end{align*}
Thus, $T_v(\mathbb D)$ and $T_v(\mathbb D^*)$ are endowed with the complex Banach manifold structures, 
respectively (see \cite[Theorem 5.2]{SW}), and thus, 
$T_v(\mathbb D) \times T_v(\mathbb D^*)$ is endowed with the product complex Banach manifold structure. 
Hence, by the identification ${\rm SS}_{\mathbb C}(\mathbb S) \cong T_v(\mathbb D) \times T_v(\mathbb D^*)$ in
Proposition \ref{id}, we may regard ${\rm SS}_{\mathbb C}(\mathbb S)$ as a domain 
of $\mathcal B_0(\mathbb D^*) \times \mathcal B_0(\mathbb D)$. 

Let us recall some results on $\rm VMOA$ space (see \cite{Ga, Pom, Zh}). It is known that $\phi \in {\rm VMOA}(\mathbb D^*)$ if $\phi \in H^2(\mathbb D^*)$ and if the boundary values $b(\phi)$ of 
$\phi$ on $\mathbb S$ belongs to ${\rm VMO}(\mathbb S)$. Moreover, ${\rm VMOA}(\mathbb D^*)$ is a Banach space modulo constants,
and the norm defined by $\Vert \phi \Vert = \Vert b(\phi) \Vert_{*}$ is equivalent to the norm defined by the Carleson measure
$\lambda_\phi$ given by \eqref{CML}. 
Then, the boundary extension $b: {\rm VMOA}(\mathbb D^*) \to {\rm VMO}(\mathbb S)$ is regarded as a linear isometric operator. Let 
$$
\mathcal T_v(\mathbb D^*) = \{\mathcal L_{G(\mu, 0)} \in {\rm VMOA}(\mathbb D^*) \mid \mu \in \mathcal M_0(\mathbb D)\}. 
$$
By the normalization of the VMO-quasiconformal mapping $G(\mu, 0)$ relative to
$\mathbb S$, the map
$$
\alpha: \, \mathcal T_v(\mathbb D^*) \to \beta (T_v(\mathbb D))
$$
is bijective by sending $\mathcal L_{G(\mu, 0)}$ to $\mathcal S_{G(\mu, 0)}$. Moreover, 
from the arguments in \cite[Section 6]{SW}, we see that $\alpha$ is a biholomorphic homeomorphism.

By Proposition \ref{curve}, we can consider a map $L:{\rm SS}_{\mathbb C}(\mathbb S) \to {\rm VMO}(\mathbb S)$ defined 
by $L(\gamma)=\log \gamma'$. 
Then, with respect to the complex structure of ${\rm SS}_{\mathbb C}(\mathbb S)$ given as above,
we see the following:

\begin{theorem}\label{holo}
The map $L:{\rm SS}_{\mathbb C}(\mathbb S) \to {\rm VMO}(\mathbb S)$ is holomorphic.
\end{theorem}

\begin{proof}
We will prove that $L$ is holomorphic at any point $\gamma=G(\mu_1,\mu_2)|_{\mathbb S}$ in 
${\rm SS}_{\mathbb C}(\mathbb S)$.
Since ${\rm SS}_{\mathbb C}(\mathbb S)$ can be regarded as a domain of the product 
$\mathcal B_0(\mathbb D^*) \times \mathcal B_0(\mathbb D)$ of the Banach spaces, 
the Hartogs theorem for Banach spaces (see \cite[Theorem 14.27]{Ch} and \cite[Theorem 36.8]{Mu}) implies that
we have only to prove that $L$ is separately holomorphic. Thus, 
by fixing $[\mu_2] \in T_v(\mathbb D^*)$, we will show that 
$\log (G(\mu,\mu_2)|_{\mathbb S})' \in {\rm VMO}(\mathbb S)$ depends holomorphically on $[\mu] \in T_v(\mathbb D)$.
The other case is similarly treated.

By the proof of Proposition \ref{curve}, we have
$$
\log (G(\mu,\mu_2)|_{\mathbb S})' =\log h' \circ f + \log f',
$$
where $f: \mathbb S \to \mathbb S$ is the boundary extension of 
the fixed bi-Lipschitz diffeomorphism $F:\mathbb D \to \mathbb D$
defined by the Douady--Earle extension for the Teichm\"uller class $[\mu_2]$, 
and $h:\mathbb S \to \mathbb C$ is the restriction of
the quasiconformal homeomorphism $H$ of $\mathbb C$ that is conformal on $\mathbb D^*$ and
has the complex dilatation $F_* \mu$ on $\mathbb D$. 
Since $F$ is a
bi-Lipschitz diffeomorphism, we conclude by Proposition \ref{right} that $F_*$ acts on ${\mathcal M}_0(\mathbb D)$
as a biholomorphic automorphism, and its action projects down to $T_v(\mathbb D)$, which is biholomorphically equivalent
to $\mathcal T_v(\mathbb D^*)$ through $\alpha$ and $\beta$, 
also as a biholomorphic automorphism. 
Hence, $\log (H|_{\mathbb D^*})'
\in \mathcal T_v(\mathbb D^*) \subset {\rm VMOA}(\mathbb D^*)$ depends on $[\mu] \in T_v(\mathbb D)$ holomorphically.

We see that the boundary extension $b:{\rm VMOA}(\mathbb D^*) \to {\rm VMO}(\mathbb S)$ is
a bounded linear operator. Moreover, by Proposition \ref{pullback}, the pull-back operator
$P_f:{\rm VMO}(\mathbb S) \to {\rm VMO}(\mathbb S)$ induced by $f \in {\rm SS}(\mathbb S)$ is also a bounded linear operator.
Therefore,
$$
\log h' \circ f=P_f \circ b ({\mathcal L}_{H|_{\mathbb D^*}}) \in {\rm VMO}(\mathbb S)
$$
in particular depends on $[\mu] \in T_v(\mathbb D)$ holomorphically, and so does 
$\log (G(\mu,\mu_2)|_{\mathbb S})'$.
\end{proof}

By Proposition \ref{id}, ${\rm SS}_{\mathbb C}(\mathbb S)$ is identified with $T_v(\mathbb D) \times T_v(\mathbb D^*)$.
Then, we can represent any element of ${\rm SS}_{\mathbb C}(\mathbb S)$ by a pair $([\mu_1],[\mu_2])$.
For $\mu_1 \in \mathcal M_0(\mathbb D)$, the same notation $\mu_1$ denotes
the Beltrami coefficient on $\mathbb D^*$ given by
the reflection \eqref{refl}, and similarly
for $\mu_2 \in \mathcal M_0(\mathbb D^*)$, $\mu_2 \in \mathcal M_0(\mathbb D)$ denotes its reflection.
Let $j$ be the anti-holomorphic involution of ${\rm SS}_{\mathbb C}(\mathbb S)$ defined by
$([\mu_1],[\mu_2]) \mapsto ([\mu_2],[\mu_1])$. Then, the fixed point locus of $j$ is
$$
{\rm SS}(\mathbb S) \cong \{([\mu],[\mu]) \in T_v(\mathbb D) \times T_v(\mathbb D^*)\},
$$
which is thus the real-analytic submanifold of ${\rm SS}_{\mathbb C}(\mathbb S)$, and the map 
$T_v(\mathbb D) \to {\rm SS}(\mathbb S)$ given by 
$[\mu] \mapsto ([\mu],[\mu])$ is a real-analytic homeomorphism onto ${\rm SS}(\mathbb S)$.

We consider the restriction of $L:{\rm SS}_{\mathbb C}(\mathbb S) \to {\rm VMO}(\mathbb S)$ to ${\rm SS}(\mathbb S)$.
We also look at this by composing the real analytic equivalence $T_v(\mathbb D) \cong {\rm SS}(\mathbb S)$
and taking the real part of ${\rm VMO}(\mathbb S)$. The resulting map is 
$$
\widetilde L:T_v(\mathbb D) \to {\rm VMO}_{\mathbb R}(\mathbb S),
$$
which sends a strongly symmetric homeomorphism $h$ of $\mathbb S$ to $\log |h'|$.

\begin{corollary}\label{holo2}
Under the identification $T_v(\mathbb D) \cong {\rm SS}(\mathbb S)$, the map 
$\widetilde L:T_v(\mathbb D) \to {\rm VMO}_{\mathbb R}(\mathbb S)$ sending $h \in {\rm SS}(\mathbb S)$ to
$\log |h'| \in {\rm VMO}_{\mathbb R}(\mathbb S)$ is real-analytic.
\end{corollary}

This follows immediately from Theorem \ref{holo}. As stated in Corollary \ref{mainss},
$\widetilde L$ is in fact a real-analytic homeomorphism whose inverse $\widetilde L^{-1}$ is also real-analytic.

\section{The variant of Beurling--Ahlfors extension by heat kernel}\label{BAheat}

In order to prove the second statement of Theorem \ref{mains}, we need the variant of Beurling--Ahlfors extension by  heat kernel introduced in \cite{FKP}. In this section, we recall the definitions and results on this extension. 

We begin with the classical Beurling--Ahlfors extension given in \cite{BA}. 
Let 
$\phi(x)=\frac{1}{2} 1_{[-1,1]}(x)$ and $\psi(x)=\frac{r}{2} 1_{[-1,0]}(x)+\frac{-r}{2} 1_{[0,1]}(x)$
for some $r>0$, where $1_E$ denotes the characteristic function of $E \subset \mathbb R$. For any function $\varphi(x)$ 
on $\mathbb R$ and for $y>0$, we set $\varphi_y(x)=y^{-1} \varphi(y^{-1}x)$.
Then, for a quasisymmetric homeomorphism $f$ of the real line $\mathbb R$, the Beurling--Ahlfors extension
$F(x,y)=(U(x,y),V(x,y))$ for $(x,y) \in \mathbb U$  defined by the convolutions
$U(x,y)=(f \ast \phi_y)(x)$ and $V(x,y)=(f \ast \psi_y)(x)$ is a quasiconformal homeomorphism of 
the upper half-plane $\mathbb U$ onto itself. 
Modification and variation to the Beurling--Ahlfors extension have been made by changing the functions
$\phi$ and $\psi$. 

For a complex-valued function $u$ on $\mathbb R$ 
such that $u \in \rm BMO(\mathbb R)$ in general, we consider
a mapping $\gamma_u=\gamma: \mathbb{R} \to \mathbb{C}$ given by
\begin{equation}\label{gamma}
\gamma(x) = \gamma(0) + \int_0^x e^{u(t)} dt.
\end{equation}
Let $\phi(x)=\frac{1}{\sqrt \pi}e^{-x^2}$ and $\psi(x)=\phi'(x)=-2x \phi(x)$.
Then, we extend $\gamma$ to 
$\mathbb U$ as a differentiable map $F_u=F: \mathbb{U} \to \mathbb{C}$ defined by 
\begin{equation}\label{F}
\begin{split}
&F(x, y) = U(x, y) + iV(x, y);\\
U(x,y)&=(\gamma \ast \phi_y)(x),\ V(x,y)=(\gamma \ast \psi_y)(x).
\end{split}
\end{equation}

The partial derivatives of
$U$ and $V$ can be represented as follows:
\begin{align*}
U_x(x,y)&=\frac{\partial U}{\partial x}=(e^u \ast \phi_y)(x);\\
V_x(x,y)&=\frac{\partial V}{\partial x}=(e^u \ast \psi_y)(x);\\
U_y(x,y)&=\frac{\partial U}{\partial y}=(\gamma \ast \frac{\partial \phi_y}{\partial y})(x)
=\frac{1}{2}(e^u \ast \psi_y)(x)=\frac{1}{2}V_x(x,y);\\
V_y(x,y)&=\frac{\partial V}{\partial y}=(\gamma \ast \frac{\partial \psi_y}{\partial y})(x)
=U_x(x,y)+\frac{y^2}{2}(e^u \ast (\phi_y)'')(x),
\end{align*} 
where we have used
\begin{align*}
\frac{\partial \phi_y}{\partial y}=\frac{y}{2}\frac{\partial^2 \phi_y}{\partial x^2} =\frac{1}{2}(\psi_y)'; \quad 
\frac{\partial \psi_y}{\partial y}=\frac{\partial \phi_y}{\partial x}+\frac{y^2}{2} \frac{\partial^3 \phi_y}{\partial x^3}
=(\phi_y)'+\frac{y^2}{2}(\phi_y)'''.
\end{align*}
In particular, each of 
$U_y(x,y)$, $V_x(x,y)$, and $(U_x-V_y)(x,y)$ can be represented by the convolution $(e^u \ast a_y)(x)$ explicitly
for a certain real-valued function $a \in C^{\infty}(\mathbb R)$
such that $\int_{\mathbb R} a(x)dx=0$, $|a(x)|$ is an even function, and $a(x)=O(x^2e^{-x^2})$ $(|x| \to \infty)$. 
For instance, 
$V_x(x,y)=(e^u \ast \psi_y)(x)$
for $\psi(x)=-\frac{2}{\sqrt \pi}xe^{-x^2}$. 

Next, we consider the complex derivatives 
\begin{align*}
F_{\bar z}(x,y) &= \frac{1}{2}(F_x + iF_y) = \frac{1}{2}\left((U_x - V_y) + iU_y + iV_x \right);\\
F_z(x,y) &= \frac{1}{2}(F_x - iF_y) = U_x+\frac{1}{2}\left(-(U_x- V_y) - iU_y+ iV_x  \right). 
\end{align*}
From these expressions, we can find two complex-valued functions $\alpha, \beta \in C^{\infty}(\mathbb R)$ 
independent of $u$ such that
\begin{equation}\label{alphabeta}
F_{\bar{z}}=e^u \ast \alpha_y(x), \quad F_{z}=e^u \ast \beta_y(x),
\end{equation}
and $\int_{\mathbb{R}}\alpha(x)dx = 0$, $\alpha(x)=O(x^2e^{-x^2})$ $(|x| \to \infty)$, 
$\int_{\mathbb{R}}\beta(x)dx = 1$, $\beta(x)=O(x^2e^{-x^2})$ $(|x| \to \infty)$.
In particular, we can assume that
$$
|\alpha(x)| \leq Ce^{-|x|}, \quad |\beta(x)| \leq Ce^{-|x|}
$$
for some constant $C>0$. 
We set $\mu_u(x,y)=F_{\bar z}/F_z$, and call it the complex dilatation of $F$ even though the map $F=F_u$ 
defined by $(\ref{F})$ is not necessarily a quasiconformal homeomorphism. 

In the case where $u$ is a real-valued function such that $e^u$ is
an $A_\infty$-weight, the situation becomes simpler.
In this case, the extension $F:\mathbb U \to \mathbb U$ of $\gamma:\mathbb R \to \mathbb R$ is
the variant of Beurling--Ahlfors extension by heat kernel.
The following results are obtained in \cite[Theorem 4.2]{FKP} and \cite{WM-2, WM-4}.

\begin{theorem}[]\label{FKP}
For a real-valued $u \in {\rm BMO}_{\mathbb R}(\mathbb R)$ such that $e^u$ is an $\rm A_\infty$-weight, 
the map $F_u$ given by
the variant of Beurling--Ahlfors extension by heat kernel is a quasiconformal
diffeomorphism of $\mathbb U$ onto itself whose complex dilatation $\mu_u$ belongs to $\mathcal M(\mathbb U)$.
Moreover, it is bi-Lipschitz with respect to the hyperbolic metric. 
If $u \in {\rm VMO}_{\mathbb R}(\mathbb R)$ in addition, then $\mu_u \in \mathcal M_0(\mathbb U)$.
\end{theorem}

We note that a Carleson measure $\lambda$ on $\mathbb U$ can be defined by the condition
$$
\Vert \lambda \Vert_c = \sup \frac{\lambda(I \times (0,|I|))}{|I|}<\infty,
$$
where the supremum is taken over all bounded intervals $I$ on $\mathbb R$. A vanishing Carleson measure and 
the spaces of the Beltrami coefficients $\mathcal M(\mathbb U)$ and $\mathcal M_0(\mathbb U)$ are defined in the same way as
in the case of $\mathbb D$.

\section{Local boundedness of the complex dilatations}
In this section,
we prepare the essential step for the proof of the second part of Theorem \ref{mains}. 
This is based on the variant of Beurling-Ahlfors extension by heat kernel reviewed in Section 4.
We will show the local boundedness of the complex dilatations of these extensions of certain strongly symmetric embeddings.

First, we state the following two claims on ${\rm BMO}(\mathbb R)$ given in \cite[Proposition 2.4, Lemma 2.6]{WM-4} 
which play key roles in the next two lemmas. Let $I(x, y) \subset \mathbb R$ be 
the interval $(x-y, x+y)$ for any $x \in \mathbb R$ and $y > 0$. 

\begin{proposition}\label{prep1}
Let $u$ and $\varphi$ be 
complex-valued functions on $\mathbb R$ 
such that $u \in \rm BMO(\mathbb R)$ and $|\varphi(x)| \leq Ce^{-|x|}$ for some constant $C>0$.
Let $k \in \mathbb N$ be a positive integer.
Then,
$$
\int_{\mathbb{R}}|\varphi_y(x-t)||u(t)-u_{I(x,y)}|^kdt \leq C(k) \Vert u \Vert_*^k
$$
for a positive constant $C(k)$ depending on $k$.
\end{proposition}

\begin{proposition}\label{nearreal}
Let $\varphi$ be a 
complex-valued function on $\mathbb R$ 
such that $|\varphi(x)| \leq Ce^{-|x|}$ for some constant $C>0$. 
For each real-valued essentially bounded function $u_0 \in L^\infty_{\mathbb R}(\mathbb R)$,
there exists a positive constant $C(u_0)$ depending on $\Vert u_0 \Vert_{\infty}$  such that every 
complex-valued function $u \in {\rm BMO}(\mathbb R)$ with $\Vert u-u_0 \Vert_{*} \leq C_{JN}/2$
satisfies 
$$
\int_{\mathbb{R}}|\varphi_y(x-t)|e^{|u(t)-u_{I(x,y)}|}dt \leq C(u_0).
$$
\end{proposition}

In the following two lemmas, we prove the local boundedness of the complex dilatation $\mu_w$ of 
the variant of Beurling--Ahlfors extension by heat kernel given by $w \in {\rm BMO}(\mathbb R)$
when $w$ is sufficiently close to $L^\infty_{\mathbb R}(\mathbb R)$. 
Hereafter, we use the comparability symbols $\asymp$ and $\lesssim$ instead of mentioning the existence 
of some constants that are uniform with respect to the obvious quantities in the situation.

\begin{lemma}\label{bounded}
For any real-valued bounded function $u_0 \in L^{\infty}_{\mathbb R}(\mathbb R)$, there are $\delta_1>0$ and $M_1>0$ such that
if $w \in {\rm BMO}(\mathbb R)$ satisfies
$\Vert w-u_0 \Vert_{*} <\delta_1$, then $\Vert \mu_{w} \Vert_{\infty} \leq M_1$.
\end{lemma}

\begin{proof}
For any $w \in {\rm BMO}(\mathbb R)$, 
we set $w=u+iv$ where $u$ and $v$ are real-valued.
Fixing $I(x,y)=(x-y,x+y) \subset \mathbb R$ for $(x,y) \in \mathbb U$,
we consider
\begin{equation}\label{mu-fraction}
|\mu_{w}(x,y)|=\frac{|\alpha_y\ast e^{w}(x)|}{|\beta_y\ast e^{w}(x)|}
=\frac{|\alpha_y\ast e^{w- w_{I(x,y)}}(x)|}{|\beta_y\ast e^{w-w_{I(x,y)}}(x)|}.
\end{equation}
The denominator is estimated from below as
\begin{equation*}
\begin{split}
|\beta_y\ast e^{w-w_{I(x,y)}}(x)|&=|\beta_y\ast(e^{u-u_{I(x,y)}}(e^{i(v-v_{I(x,y)})}-1))(x)
+\beta_y\ast e^{u-u_{I(x,y)}}(x)|\\
&\geq |\beta_y\ast e^{u-u_{I(x,y)}}(x)|-|\beta_y\ast(e^{u-u_{I(x,y)}}(e^{i(v-v_{I(x,y)})}-1))(x)|.
\end{split}
\end{equation*}
Here, since $u$ is real-valued, we see that
\begin{equation*}
\begin{split}
|\beta_y\ast e^{u-u_{I(x,y)}}(x)| &\asymp 
|\phi_y\ast e^{u-u_{I(x,y)}}(x)|\\
 & \geq \int_{\mathbb{R}}\phi_y(x-t) e^{-|u(t)-u_{I(x,y)}|} dt\\
& \geq \frac{2}{\sqrt \pi e}\left(\frac{1}{2y}\int_{|x-t|<y} e^{-|u(t)-u_{I(x,y)}|} dt \right)\\ 
& \geq \frac{2}{\sqrt \pi e}\exp (-\Vert u \Vert_*) 
\end{split}
\end{equation*}
for $\phi(x)=\frac{1}{\sqrt \pi}e^{-x^2}$. 
On the contrary, the Cauchy--Schwarz inequality and
$|e^{ix}-1| \leq |x|$
yield that
\begin{equation}\label{denomi}
\begin{split}
&\quad |\beta_y\ast(e^{u-u_{I(x,y)}}(e^{i(v-v_{I(x,y)})}-1))(x)| \\
&\leq \int_{\mathbb R} |\beta_y(x-t)|e^{u(t)-u_{I(x,y)}}|e^{i(v(t)-v_{I(x,y)})}-1|dt \\
&\leq \left(\int_{\mathbb R} |\beta_y(x-t)|e^{2(u(t)-u_{I(x,y)})}dt \right)^{1/2}
\left(\int_{\mathbb R} |\beta_y(x-t)||v(t)-v_{I(x,y)}|^2 dt \right)^{1/2}.
\end{split}
\end{equation}
By Proposition \ref{nearreal}, the first factor of the last line of (\ref{denomi}) is bounded if $\Vert u-u_0 \Vert_{*}$ is
sufficiently small.
By Proposition \ref{prep1}, the second factor is bounded by a multiple of $\Vert v \Vert_*$.
Thus, the denominator in the fraction of (\ref{mu-fraction}) is
bounded away from $0$ if $\Vert v \Vert_{*}$ is sufficiently small.
In particular, there is some $0 < \delta_1 < C_{JN}/4$  such that if $\Vert u-u_0 \Vert_{*} < \delta_1/2$ and
if $\Vert v \Vert_{*} < \delta_1/2$, then the denominator is uniformly bounded away from $0$. 
Hence, there is some constant $A>0$ such that if $\Vert w-u_0 \Vert_{*} <\delta_1$ then
\begin{equation}\label{onlyalpha}
|\mu_{w}(x,y)| \leq A |\alpha_y\ast e^{w-w_{I(x,y)}}(x)|.
\end{equation}

In \eqref{onlyalpha}, $|\alpha_y\ast e^{w-w_{I(x,y)}}(x)|$ coincides with
$|\alpha_y\ast (e^{w-w_{I(x,y)}}-1)(x)|$ by $\int_{\mathbb R} \alpha(x)dx=0$. Then, similarly to the above estimate,  the Cauchy--Schwarz inequality and $|e^z - 1| \leq |z|e^{|z|}$ yield that
\begin{equation}\label{numer}
\begin{split}
&\quad |\alpha_y\ast (e^{w-w_{I(x,y)}}-1)(x)|\\
&\leq \int_{\mathbb R} |\alpha_y(x-t)||w(t)-w_{I(x,y)}|e^{|w(t)-w_{I(x,y)}|}dt\\
&\leq \left(\int_{\mathbb R} |\alpha_y(x-t)||w(t)-w_{I(x,y)}|^2dt \right)^{1/2}
\left(\int_{\mathbb R} |\alpha_y(x-t)|e^{2|w(t)-w_{I(x,y)}|}dt \right)^{1/2}.
\end{split}
\end{equation}
Again by Propositions \ref{prep1} and \ref{nearreal}, this is bounded if
$\Vert w-u_0 \Vert_{*} \leq C_{JN}/4$.
Thus, we see that $\Vert \mu_{w} \Vert_\infty \leq M_1$ for some
constant $M_1>0$ if
$\Vert w-u_0 \Vert_{*} <\delta_1$. 
\end{proof}

In the next lemma, the assumption on $\Vert w-u_0 \Vert_{*}$ is necessary for using lemma \ref{bounded}
at the beginning. In the rest of the proof, only the smallness of $\Vert u-u_0 \Vert_{*}$ is required.

\begin{lemma}\label{mainlemma}
For any real-valued bounded function $u_0 \in L^{\infty}_{\mathbb R}(\mathbb R)$, there are $\delta_2>0$ and $M_2>0$ such that
if $w \in {\rm BMO}(\mathbb R)$ satisfies
$\Vert w-u_0 \Vert_{*} <\delta_2$, then $\Vert \mu_{w} \Vert_{c} \leq M_2$.
\end{lemma}

\begin{proof}
By inequality (\ref{onlyalpha}) in the proof of Lemma \ref{bounded}, we have
\begin{equation}\label{no1}
|\mu_w(x,y)|\lesssim |\alpha_y\ast e^{w-w_{I(x,y)}} (x)|=\frac{|\alpha_y\ast e^{w} (x)|}{e^{u_{I(x,y)}}}
\end{equation}
for $w=u+iv$ if $\Vert w-u_0 \Vert_{*}$ is sufficiently small.
We also note that if $\Vert u-u_0 \Vert_{*}$ is sufficiently small,
then $e^{ru}$ is an $A_\infty$-weight for $r \in [-4,4]$ (see \cite[Corollary IV.5.15]{GR}), 
and moreover, the $A_\infty$-constant $C_\infty(e^{ru})$ for $e^{ru}$  depends only on $u_0$
(see also the proof of \cite[Lemma 2.6]{WM-4}). 
By the property of $A_\infty$-weights in \eqref{iff} and the Cauchy--Schwarz inequality, we have
\begin{equation}\label{no2}
\frac{|\alpha_y\ast e^{w} (x)|}{e^{u_{I(x,y)}}}\\
\leq \frac{C_\infty(e^u)|\alpha_y\ast e^{w} (x)|}{(e^u)_{I(x,y)}}
\leq C_\infty(e^u)|\alpha_y\ast e^{w} (x)|(e^{-u})_{I(x,y)}
\end{equation}
for $I(x,y)=(x-y,x+y)$.

Using inequalities (\ref{no1}) and (\ref{no2}), 
for any bounded interval $I=I(x_0,t/2)$ for $x_0 \in \mathbb R$ and $t>0$, we have
\begin{equation*}
\begin{split}
&\quad \frac{1}{|I|}\int_0^{|I|}\!\int_{I}|\mu_w(x,y)|^2 \frac{dxdy}{y}\\
&\lesssim \frac{C_\infty(e^u)^2}{|I|}\int_0^{|I|}\!\int_{I}|\alpha_y\ast e^{w} (x)|^2((e^{-u})_{I(x,y)})^2 \frac{dxdy}{y}\\
&=\frac{C_\infty(e^u)^2}{t}\int_{I(x_0,t/2)}\! \int_0^t
\left(\frac{1}{|I(x,y)|}\int_{I(x,y)}e^{-u(s)}ds \right)^2 |\alpha_y\ast e^{w} (x)|^2 \frac{dy}{y} dx\\
&=\frac{C_\infty(e^u)^2}{t}\int_{I(x_0,t/2)}\! \int_0^t
\left(\frac{1}{2y}\int_{|x-s|<y}e^{-u(s)}1_{I(x_0,N_0t)}(s)ds \right)^2 |\alpha_y\ast e^{w} (x)|^2 \frac{dy}{y} dx.
\end{split}
\end{equation*}
Here, we have replaced $e^{-u(s)}$ with the product $e^{-u(s)}1_{I(x_0,N_0t)}(s)$ of the characteristic function
by a fact that
if $x \in I(x_0,t/2)$ and $|x -s| < y \leq t$ then $s \in I(x_0,N_0t)$ for any $N_0 \geq 2$.
We will choose $N_0$ sufficiently large and fix it later.
The last line above is estimated further as follows.
\begin{equation}\label{routine}
\begin{split}
&\quad\frac{C_\infty(e^u)^2}{t}\int_{I(x_0,t/2)}\! \int_0^t
\left(\frac{1}{2y}\int_{|x-s|<y}e^{-u(s)}1_{I(x_0,N_0t)}(s)ds \right)^2 |\alpha_y\ast e^{w} (x)|^2 \frac{dy}{y} dx\\
&\leq \frac{C_\infty(e^u)^2}{t}\int_{I(x_0,t/2)} \sup_{y >0} 
\left(\frac{1}{2y}\int_{|x-s|<y}e^{-u(s)}1_{I(x_0,N_0t)}(s)ds \right)^2 
\left(\int_0^t |\alpha_y\ast e^{w} (x)|^2 \frac{dy}{y}\right) dx\\
&\leq C_\infty(e^u)^2\left[\frac{1}{t}\int_{I(x_0,t/2)} \left(\sup_{y >0} \left\{ 
\frac{1}{2y}\int_{|x-s|<y}e^{-u(s)}1_{I(x_0,N_0t)}(s)ds\right\} \right)^4 dx \right]^{1/2}\\
&\qquad \qquad \qquad \qquad \times \left[\frac{1}{t}\int_{I(x_0,t/2)} \left(\int_0^t |\alpha_y\ast e^{w} (x)|^2 \frac{dy}{y}\right)^2 dx \right]^{1/2}. 
\end{split}
\end{equation}

First, we consider the first integral factor of the last line of (\ref{routine}).
Here, we use the $L^4$-boundedness of the Hardy--Littlewood maximal function
$$
M(e^{-u}1_{I(x_0,N_0t)})(x) =\sup_{y >0} 
\left(\frac{1}{2y}\int_{|x-s|<y}e^{-u(s)}1_{I(x_0,N_0t)}(s)ds\right).
$$
Then, there is a constant $C_1>0$ such that
\begin{equation}\label{C_1}
\begin{split}
&\quad \frac{1}{t}\int_{I(x_0,t/2)} \left(\sup_{y >0} \left\{ 
\frac{1}{2y}\int_{|x-s|<y}e^{-u(s)}1_{I(x_0,N_0t)}(s)ds\right\} \right)^4 dx \\
&\leq \frac{1}{t}\int_{\mathbb R} \left(\sup_{y >0} \left\{ 
\frac{1}{2y}\int_{|x-s|<y}e^{-u(s)}1_{I(x_0,N_0t)}(s)ds\right\} \right)^4 dx \\
&= \frac{1}{t}\int_{\mathbb R} \left( M(e^{-u}1_{I(x_0,N_0t)})(x)\right)^4 dx
\leq \frac{C_1}{t}\int_{I(x_0,N_0t)}e^{-4u(s)} ds.
\end{split}
\end{equation}
For the estimate of the last term, we recall that, if $\Vert u-u_0 \Vert_{*}$ is sufficiently small, then 
$e^{-4u}$ is an $A_\infty$-weight. In addition, throughout the estimate of the complex dilatation $\mu_w(x,y)$,
we may assume that $u_{I(x_0,N_0t)}=0$ for the given and fixed constants $x_0$, $t$ and $N_0$ 
because adding a constant to $u$ corresponds to a dilation of $F$, which does not change the complex dilatation $\mu_{w}$. 
(When $x_0$, $t$, or $N_0$ change, we regard that the assumption is renewed accordingly.) 
These facts yield that
\begin{equation}\label{e-4u}
\begin{split}
&\quad \frac{1}{|I(x_0,N_0t)|}\int_{I(x_0,N_0t)}e^{-4u(s)} ds\\
&=\frac{1}{|I(x_0,N_0t)|}\int_{I(x_0,N_0t)}e^{-4(u(s)-u_{I(x_0,N_0t)})} ds\\
&\leq C_\infty(e^{-4u}) \exp \left(\frac{4}{|I(x_0,N_0t)|}\int_{I(x_0,N_0t)}|u(s)-u_{I(x_0,N_0t)}|ds \right)
\leq C_\infty(e^{-4u})e^{4\Vert u \Vert_*}.
\end{split}
\end{equation}
Therefore, by (\ref{C_1}) and (\ref{e-4u}), we have
\begin{equation*}
\begin{split}
\frac{1}{t}\int_{I(x_0,t/2)} \left(\sup_{y >0} \left\{ 
\frac{1}{2y}\int_{|x-s|<y}e^{-u(s)}1_{I(x_0,N_0t)}(s)ds\right\} \right)^4 dx 
\leq 2C_1N_0C_\infty(e^{-4u})e^{4\Vert u \Vert_*},
\end{split}
\end{equation*}
which shows the local boundedness of the first integral factor of the last line of (\ref{routine}).

Next, we consider the second integral factor of the last line of (\ref{routine}).
In the following,
we decompose it further by dividing $\mathbb R$ into the interval $I(x_0,Nt)$ and its complement $I(x_0,Nt)^c$
for a sufficiently large $N>0$:
\begin{equation}\label{N-divide}
\begin{split}
&\quad \frac{1}{t}\int_{I(x_0,t/2)} \left(\int_0^t |\alpha_y\ast e^{w} (x)|^2 \frac{dy}{y}\right)^2 dx\\
&\leq \frac{8}{t} \int_{I(x_0,t/2)}\left(\int_0^t  |\alpha_y\ast (e^{w}1_{I(x_0,Nt)})(x)|^2 \frac{dy}{y}\right)^2dx\\
&\qquad\qquad+\frac{8}{t}\int_{I(x_0,t/2)}\left(\int_0^t |\alpha_y\ast (e^{w}1_{I(x_0,Nt)^c})(x)|^2 \frac{dy}{y} \right)^2dx.
\end{split}
\end{equation}

We consider the first term of the right side of inequality (\ref{N-divide}).
We utilize the Littlewood--Paley operator defined by the rapidly decreasing function $\alpha$ with $\int_{\mathbb R} \alpha(x)dx=0$:
$$
S_\alpha(e^{w}1_{I(x_0,Nt)})(x)
=\left( \int_0^\infty  |\alpha_y\ast (e^{w}1_{I(x_0,Nt)})(x)|^2 \frac{dy}{y}\right)^{1/2}. 
$$
By the $L^4$-boundedness of the Littlewood--Paley operator $S_\alpha$ (see \cite[p.363]{BCP}), 
we have the estimate of the first term: there is some constant $C_2>0$ such that
\begin{equation}\label{firstterm}
\begin{split}
&\quad \frac{8}{t}\int_{I(x_0,t/2)}\left(\int_0^t  |\alpha_y\ast (e^{w}1_{I(x_0,Nt)})(x)|^2 \frac{dy}{y}\right)^2dx\\
& \leq
\frac{8}{t}\int_{\mathbb R} \left( \int_0^\infty  |\alpha_y\ast (e^{w}1_{I(x_0,Nt)})(x)|^2 \frac{dy}{y}\right)^2 dx
=  \frac{8}{t}\int_{\mathbb R} (S_\alpha(e^{w}1_{I(x_0,Nt)})(x))^4 dx\\
& \leq \frac{8C_2}{t}\int_{\mathbb R} |e^{w(x)}|^4 1_{I(x_0,Nt)}(x) dx 
= \frac{8C_2}{t} \int_{I(x_0,Nt)} e^{4u(x)}dx.
\end{split}
\end{equation}

We consider the second term of the right side of inequality (\ref{N-divide}). We have an inequality
\begin{equation*}
\begin{split}
|\alpha_y\ast (e^{w}1_{I(x_0,Nt)^c})(x)|
&\leq \int_{I(x_0,Nt)^c}|\alpha_y(x-s)|e^{u(s)} ds
\end{split}
\end{equation*}
by assuming $x \in I(x_0,t/2)$. 
For this estimate, we use a fact that the $A_\infty$-weight $e^u$ has the doubling property
with some constant $\rho=\rho(e^u)>1$ as in \eqref{doubling}. This constant can be estimated in terms of
the $A_\infty$-constant $C_\infty(e^u)$, and hence this depends only on $u_0$ when $\Vert u-u_0 \Vert_*$ is
sufficiently small.
We also note that $|\alpha|$ is an even function. 

Let $n_0=n_0(y,t,N) \in \mathbb N$ satisfy $2^{n_0-1}=(N-1)t/y$ (we may adjust $N$ so that $n_0$ becomes an integer).
Then, for $x \in I(x_0, t/2)$, we see that
\begin{align*}
\int_{I(x_0,Nt)^c}|\alpha_y(x-s)|e^{u(s)} ds
&\leq \int_{|s-x_0| \geq (N-1)t} e^{u(s)} |\alpha_y(s)|ds\\
&=\sum_{n=n_0}^\infty \int_{2^{n-1}y \leq |s-x_0| < 2^ny} e^{u(s)} |\alpha_y(s)|ds\\
&\leq \sum_{n=n_0}^\infty \rho^n \vert\alpha_y(2^{n-1}y)\vert \int_{|s-x_0| < y}e^{u(s)}ds.
\end{align*}
Here, by $|\alpha(x)| \leq Ce^{-|x|}$, we have $\rho^n \vert\alpha_y(2^{n-1}y)\vert \leq C\rho^n e^{-2^{n-1}}/y$.
For $n \geq n_0(y,t,N)$, we may assume that $\rho^ne^{-2^{n-2}} \leq 1$. In fact, 
by $2^{n_0-1} \geq N-1$, this holds when $N$ is sufficiently large.
Moreover, if $\Vert u - u_0 \Vert_{*} < C_{JN}/2$, then the John--Nirenberg inequality \eqref{JN} implies that 
\begin{equation*}
\begin{split}
\int_{|s-x_0| < y} e^{u(s)} ds & \leq \int_{I(x_0,Nt)}e^{u(s)}ds = \int_{I(x_0,Nt)}e^{u(s) - u_{I(x_0, Nt)}}ds\\
&\leq \frac{2Nte^{2\Vert u_0 \Vert_{\infty}}}{|I(x_0, Nt)|}\int_{I(x_0, Nt)} e^{|u(s) - u_0(s) - u_{I(x_0, Nt)} + (u_0)_{I(x_0, Nt)}|} dt\\
& \leq 2Nte^{2\Vert u_0 \Vert_{\infty}} (2C_0 + 1) = DNt
\end{split}
\end{equation*}
for some constant $D>0$ depending on $\Vert u_0 \Vert_\infty$.  
Therefore, we obtain that if $x \in I(x_0, t/2)$ then
\begin{align*}
|\alpha_y\ast (e^{w}1_{I(x_0,Nt)^c})(x)|
&\leq \frac{CDNt}{y}\sum_{n=n_0}^\infty \frac{\rho^n}{e^{2^{n-2}}}\frac{1}{e^{2^{n-2}}}
\leq  \frac{CDNt}{y}\exp\left(-\frac{Nt}{4y}\right).
\end{align*}

Hence, for $x \in I(x_0,t/2)$,
\begin{equation*}
\begin{split}
\int_0^t |\alpha_y\ast (e^{w}1_{I(x_0,Nt)^c})(x)|^2 \frac{dy}{y} \leq
(CDNt)^2 \int_0^t \exp\left(-\frac{Nt}{2y}\right)\frac{dy}{y^3} \leq (CD)^2N^2 e^{-N/2}.
\end{split}
\end{equation*}
From this, we obtain that
$$
\frac{8}{t} \int_{I(x_0,t/2)} \left(\int_0^t |\alpha_y\ast (e^{w}1_{I(x_0,Nt)^c})(x)|^2 \frac{dy}{y} \right)^2dx
\to 0
$$
as $N \to \infty$ uniformly for $x_0 \in \mathbb R$ and $t>0$.

By the above argument, we can make the second term of the right side of inequality (\ref{N-divide})
arbitrarily small if we choose $N$ sufficiently large. We fix such an $N$ as $N_0$. Then, we return to the estimate of
the first term of the right side of inequality (\ref{N-divide}) as in (\ref{firstterm}).
By the assumption $u_{I(x_0,N_0t)}=0$, we have
\begin{equation}\label{e^u}
\begin{split}
&\quad \frac{8C_2}{t} \int_{I(x_0,N_0t)} e^{4u(x)}dx\\
&=\frac{16C_2N_0}{|I(x_0,N_0t)|} \int_{I(x_0,N_0t)} e^{4(u(x)-u_{I(x_0,N_0t)})}dx\\
&\leq 
16C_2 N_0 C_\infty(e^{4u}) \exp \left( \frac{4}{|I(x_0,N_0t)|} \int_{I(x_0,N_0t)}|u(x)-u_{I(x_0,N_0t)}| dx \right)\\
&\leq 16C_2 N_0 C_\infty(e^{4u})e^{4\Vert u \Vert_*}.
\end{split}
\end{equation}
This shows the local boundedness of the second integral factor of the last line of (\ref{routine}).  

The required constants $\delta_2>0$ and $M_2>0$ can be found from the above arguments. This completes the proof of Lemma \ref{mainlemma}.
\end{proof}

\section{Arguments for holomorphy and real-analycity (Theorem \ref{mains} {\rm (ii)})}
In this section, we will complete the proof of
the second part of Theorem \ref{mains}, whose precise statement is given in Theorem \ref{realana}.
In the previous section, we have seen the local boundedness of
the complex dilatations of the variant of Beurling--Ahlfors extension by heat kernel of strongly symmetric embeddings.
Our remaining task is to show the holomorphy of this dependence.
This step has already appeared in several work, but here we prove this claim from a general viewpoint.

In addition, since the variant of Beurling--Ahlfors extension by heat kernel is available on the real line but not on the unit circle,  we first prove our main result of this section in the real line (Theorem \ref{Linfty}), and then transfer it to the case of the unit circle by using lift and projection maps. 

In the theory of infinite dimensional Teichm\"uller spaces, 
there are several occasions where we have to show the holomorphy of mappings to 
function spaces $F$. Useful criteria and applications can be found in
\cite[Lemma V.5.1]{Le} (Remark: an assumption on the local boundedness of $f:U \to F$ is dropped
in its statement (i)). In the typical case where $F$ is a complex Banach space of
certain measurable functions on a domain $\Omega$ in $\mathbb C$,
we can formulate the lemma below, which is based on the general result cited above
claiming that a weak holomorphy implies the holomorphy under the local boundedness (see also \cite[p.28]{Bour}).
In fact, this kind of arguments appeared in \cite[Lemma 6.1]{ST}, and we extract essential ideas from it and
arrange them in a widely applicable form. In the case where $F$ is of the supremum norm,
we can find a similar result in \cite[Lemma 3.4]{E1}.

Let $F(\Omega)$ be a complex Banach space of measurable functions on a domain $\Omega \subset \mathbb C$.
We say that the norm $\Vert \cdot \Vert$ of this space is of {\it integral type} if 
any measurable function
$$
f:\Omega \times [0,1] \to {\mathbb C}, \qquad (z,s) \mapsto f_s(z)
$$
such that $f_s(z)$ belongs to $F(\Omega)$ for all $s \in [0,1]$ satisfies
$$
\left \Vert \int_0^1 f_s(z)dt \right \Vert \leq \sup_{0 \leq s \leq 1} \Vert f_s(z) \Vert.
$$

\begin{lemma}\label{weakholo}
Let $E$ be a complex Banach space and $F(\Omega)$ a complex Banach space of measurable functions on a domain $\Omega \subset \mathbb C$
with integral type norm $\Vert \cdot \Vert$.
Let $\Lambda:U \to F(\Omega)$ be a bounded function on an open subset $U$ of $E$
that satisfies the following property:
For any $w_0 \in U$ and any $w_1 \in E$, there is an $\epsilon>0$ such that 
$$
\lambda(\zeta)(z)=\Lambda(w_0+\zeta w_1)(z) \in \mathbb C
$$
is a holomorphic function on $\{\zeta \in \mathbb C \mid |\zeta| \leq 2\epsilon\}$
for almost every $z \in \Omega$. Then, $\Lambda$ is holomorphic on $U$.
\end{lemma}

\begin{proof}
We will show that $\Lambda$ is G\^ateaux holomorphic.
By assumption, $\lambda(\zeta)(z)$ is holomorphic on $\{\zeta \in \mathbb C \mid |\zeta| \leq 2\epsilon\}$
for almost every $z \in \Omega$. Fix such $z$.
By using the Cauchy integral formula, we have
\begin{align*}
&\quad\ \lambda(\zeta)(z) - \lambda(\zeta_0)(z) 
- (\zeta - \zeta_0) \left.\frac{d}{d\zeta}\right|_{\zeta = \zeta_0}\lambda(\zeta)(z)\\
& = \frac{1}{2\pi i}\oint_{|\tau| = 2\epsilon} \lambda(\tau)(z) \left(\frac{1}{\tau - \zeta} - \frac{1}{\tau - \zeta_0} 
- \frac{\zeta - \zeta_0}{(\tau - \zeta_0)^2}  \right)d\tau\\
&= \frac{(\zeta - \zeta_0)^2}{2\pi i}\oint_{|\tau| = 2\epsilon} \frac{\lambda(\tau)(z)}{(\tau - \zeta_0)^2(\tau - \zeta)} d\tau
\end{align*}
for any $\zeta_0$ and $\zeta$ with $|\zeta_0| <\epsilon$ and $|\zeta| <\epsilon$.
Let $M=\sup_{w \in U} \Vert \Lambda(w) \Vert<\infty$.
Then, the integral type property of the norm $\Vert \cdot \Vert$ implies that 
\begin{align*}
\left\Vert \frac{\lambda(\zeta)(z) - \lambda(\zeta_0)(z)}{\zeta - \zeta_0} - 
\left.\frac{d}{d\zeta}\right|_{\zeta = \zeta_0}\lambda(\zeta)(z) \right  \Vert
\leq \frac{2M}{\epsilon^2} |\zeta - \zeta_0|. 
\end{align*}
Consequently, the limit
$$
\lim_{\zeta\to \zeta_0} \frac{\lambda(\zeta) -  \lambda(\zeta_0)}{\zeta - \zeta_0} 
=  \left.\frac{d}{d\zeta}\right|_{\zeta = \zeta_0}\lambda(\zeta)
$$
exists in $F(\Omega)$. This shows that $\Lambda$ is G\^ateaux holomorphic on $U$. 
It is known that a locally bounded G\^ateaux holomorphic function is holomorphic in the sense that
it is Fr\'echet differentiable (see \cite[Theorem 14.9]{Ch} and \cite[Theorem 36.5]{Mu}).
Thus, we see that $\Lambda$ is holomorphic on $U$.
\end{proof}

We return to the setting in Section 5 concerning the investigation of 
the variant of Beurling--Ahlfors extension by heat kernel.
We apply the above lemma to obtain the following theorem.

\begin{theorem}\label{Linfty}
For each $u_0 \in L_{\mathbb R}^\infty(\mathbb R)$, there is a neighborhood $U(u_0)$ in ${\rm BMO}(\mathbb R)$
such that the complex dilatation $\mu_w$ of the variant of Beurling--Ahlfors extension by heat kernel
$F_w:\mathbb U \to \mathbb C$ 
belongs to $\mathcal M(\mathbb U)$ for every $w \in U(u_0)$. Moreover, the map $\Lambda_{u_0}: U(u_0) \to \mathcal M(\mathbb U)$
defined by $w \mapsto \mu_w$ is holomorphic.  This in particular implies that the map $\Lambda_{u_0}$ restricted on $U(u_0) \cap {\rm BMO}_{\mathbb R}(\mathbb R)$ is real-analytic. 
\end{theorem}

\begin{proof}
We first prove that $\Lambda_{u_0}$ is a holomorphic map to $\mathcal L(\mathbb U)$
on a neighborhood 
$$
\widetilde U(u_0)=\{w \in {\rm BMO}(\mathbb R) \mid \Vert w-u_0 \Vert_* <\delta\}
$$ 
of $u_0$, where $\delta=\min\{\delta_1,\delta_2\}>0$ is defined from Lemmas \ref{bounded} and \ref{mainlemma}.
Then, $\Lambda_{u_0}$ is bounded on $\widetilde U(u_0)$.
Moreover, we choose $\epsilon$ with
$0 < \epsilon < (\delta - \Vert w_0 \Vert_{*})/(2\Vert w_1 \Vert_{*})$
so that $w_0+\zeta w_1 \in \widetilde U(u_0)$ when $|\zeta| \leq 2\epsilon$.
Explicit representation (\ref{alphabeta}) shows that
the complex-valued function $\lambda(\zeta)(z)=\Lambda_{u_0}(w_0+\zeta w_1)(z)$ 
is holomorphic on $\{\zeta \in \mathbb C \mid |\zeta| \leq 2\epsilon\}$ with $z \in \mathbb{U}$ fixed.
To prove the holomorphy of $\Lambda_{u_0}$,
it suffices to see that the norm $\Vert \cdot \Vert_\infty+\Vert \cdot \Vert_c$ on $\mathcal L(\mathbb U)$
is of integral type by Lemma \ref{weakholo}. 

For a measurable function
$$
\mu:\mathbb U \times [0,1] \to {\mathbb C}, \qquad (z,s) \mapsto \mu_s(z)
$$
such that $\mu_s(z)$ belongs to $\mathcal L(\mathbb U)$ for all $s \in [0,1]$,
we have
$$
\left \Vert \int_0^1 \mu_s(z)ds \right \Vert_\infty 
\leq \sup_{0 \leq s \leq 1} \Vert \mu_s(z) \Vert_\infty;
$$
\begin{align*}
\left \Vert \int_0^1 \mu_s(z)ds \right \Vert_c^2&=
\sup_{x_0 \in {\mathbb R}, \ t>0}\frac{1}{t}\int_{0}^{t}\int_{I(x_0, t/2)}  
\left| \int_0^1 \mu_s(z)ds\right|^2 \frac{dxdy}{y} \\
&\leq
\sup_{x_0 \in {\mathbb R}, \ t>0}\frac{1}{t}\int_{0}^{t}\int_{I(x_0, t/2)}  
\left( \int_0^1 |\mu_s(z)|^2ds \right)\frac{dxdy}{y} \\
&\leq 
\int_0^1 \left(\sup_{x_0 \in {\mathbb R}, \ t>0} \frac{1}{t}\int_{0}^{t}\int_{I(x_0, t/2)}  
|\mu_s(z)|^2 \frac{dxdy}{y} \right)ds 
\leq \sup_{0 \leq s \leq 1} \Vert \mu_s(z) \Vert_c^2. 
\end{align*}
These inequalities show that $\Vert \cdot \Vert_\infty+\Vert \cdot \Vert_c$ is of integral type.

Thus, we obtain that $\Lambda_{u_0}:\widetilde U(u_0) \to \mathcal L(\mathbb U)$ is holomorphic.
In particular, $\Lambda_{u_0}$ is continuous there.
Since $\Lambda_{u_0}(u_0) \in \mathcal{M}(\mathbb U)$ by Theorem \ref{FKP} and 
$\mathcal{M}(\mathbb U)$ is open in $\mathcal{L}(\mathbb U)$, by taking a smaller neighborhood $U(u_0) \subset \widetilde U(u_0)$,
the image of $U(u_0)$
under $\Lambda_{u_0}$ is contained in $\mathcal{M}(\mathbb U)$. 
\end{proof}

\begin{remark}
Lemma \ref{weakholo} can be widely applied for showing holomorphy in the infinite dimensional Teichm\"uller theory
because almost all complex Banach spaces used in this theory are of integral type norms such as the spaces of
Beltrami differentials and the spaces of holomorphic differentials.
The point-wise holomorphy is also easily verified.
Proposition \ref{right} and Theorem \ref{Linfty} are two cases of all. 
\end{remark}

Finally, we transfer the result of Theorem \ref{Linfty} applied to 
${\rm VMO}_{\mathbb R}(\mathbb R) \subset {\rm BMO}_{\mathbb R}(\mathbb R)$
to the case of the unit circle $\mathbb S$. 
We consider the lift $\tilde u$ of an element $u$ of ${\rm VMO}_{\mathbb R}(\mathbb S)$ to $\mathbb R$ under
the universal covering $\mathbb R \to \mathbb S$ given by $x \mapsto e^{2\pi i x}$. Namely, we define $\tilde u(x) = u(e^{2\pi i x})$ so that it satisfies
$\tilde u(x+1)=\tilde u(x)$. 
From this periodicity,   
$\Vert u \Vert_* \leq \Vert \tilde u \Vert_* \leq 3\Vert u \Vert_*$ holds by \cite[Lemma 2.2]{Pa}, and $\tilde u$ belongs to 
${\rm VMO}_{\mathbb R}(\mathbb R)$ by $u \in {\rm VMO}_{\mathbb R}(\mathbb S)$. 
Since this correspondence ${\rm VMO}_{\mathbb R}(\mathbb S) \to {\rm VMO}_{\mathbb R}(\mathbb R)$ is a bounded linear map
by this property,
it is in particular real-analytic.

We denote the set of all such lifts by ${\rm VMO}_{\mathbb R}^{\#}(\mathbb R)$. 
Then, ${\rm VMO}_{\mathbb R}^{\#}(\mathbb R)$ is a real Banach subspace of ${\rm VMO}_{\mathbb R}(\mathbb R)$. 
Moreover, we see that
$L^{\infty}_{\mathbb R}(\mathbb R)$ is dense in ${\rm VMO}^{\#}_{\mathbb R}(\mathbb R)$ since $L^{\infty}_{\mathbb R}(\mathbb S)$ is dense in ${\rm VMO}_{\mathbb R}(\mathbb S)$ under the BMO topology induced by the norm $\Vert \cdot \Vert_{*}$ 
(see \cite {Sa}, \cite[Theorem VI.5.1]{Ga}). 
Hence, the map $\Lambda_{u_0}$, which is given locally in Theorem \ref{Linfty}, can be extended globally to some neighborhood of
${\rm VMO}^{\#}_{\mathbb R}(\mathbb R)$ by the same correspondence $w \mapsto \mu_w$. 
We denote its restriction to ${\rm VMO}^{\#}_{\mathbb R}(\mathbb R)$ by $\Lambda$. Moreover, we also consider
the closed subspace of $\mathcal{M}_0(\mathbb U)$ consisting of all elements $\tilde \mu$ satisfying the periodicity
$\tilde \mu(z+1)=\tilde \mu(z)$, which is denoted by $\mathcal{M}^{\#}_0(\mathbb U)$.

\begin{proposition}
$\Lambda$ is a real-analytic map from ${\rm VMO}^{\#}_{\mathbb R}(\mathbb R)$ to $\mathcal{M}^{\#}_0(\mathbb U)$.
\end{proposition}

\begin{proof}
By Theorem \ref{Linfty}, we see that the map $\Lambda$ from 
${\rm VMO}_{\mathbb R}^{\#}(\mathbb R)$ into $\mathcal M(\mathbb U)$ defined by $u \mapsto \mu_{u}$ is real-analytic. 
Further, by Theorem \ref{FKP} together with the period preserving property of this extension,
the image of ${\rm VMO}_{\mathbb R}^{\#}(\mathbb R)$ under 
$\Lambda$ is contained in $\mathcal M^{\#}_0(\mathbb U)$. Consequently, 
$\Lambda: {\rm VMO}_{\mathbb R}^{\#}(\mathbb R) \to \mathcal M^{\#}_0(\mathbb U)$ is real-analytic.
\end{proof}

By the periodicity of $\tilde u \in {\rm VMO}_{\mathbb R}^{\#}(\mathbb R)$ under $x \mapsto x+1$, we see that 
the variant of Beurling--Ahlfors extension by heat kernel $F_{\tilde u}$
satisfies $F_{\tilde u}(z+1)=F_{\tilde u}(z)+1$.  
Thus, $F_{\tilde u}$ can be projected to a quasiconformal homeomorphism $G_u$ of 
the punctured disk $\mathbb D\backslash \{0\}$ onto itself such that 
$G_u(e^{2\pi iz}) = e^{2\pi iF_{\tilde u}(z)}$ for $z \in \mathbb U$. 
Clearly, $G_u$ can be extended quasiconformally
to $0$, and the resulting mapping from $\mathbb D$ onto itself is still denoted by $G_u$.  
Its complex dilatation is denoted by $\nu_u$.

Summing up, for any $u \in {\rm VMO}_{\mathbb R}(\mathbb S)$, we have the quasiconformal homeomorphism $G_u$ of 
$\mathbb D$ onto itself with complex dilatation $\nu_u$. We regard this $G_u$ as
the variant of Beurling--Ahlfors extension by heat kernel transferred to the unit disk.
The map defined by the correspondence $u \mapsto \nu_u$ is denoted by $\widetilde \Lambda$.

\begin{theorem}\label{realana}
$\widetilde \Lambda$ is a real-analytic map from ${\rm VMO}_{\mathbb R}(\mathbb S)$ to $\mathcal M_0(\mathbb D)$.
\end{theorem}

\begin{proof}
The complex dilatations $\mu_{\tilde u}$ of $F_{\tilde u}$ and $\nu_u$ of $G_u$ satisfy 
$$
\nu_u(e^{2\pi i z})\overline{e^{2\pi iz}}/e^{2\pi iz} = -\mu_{\tilde u}(z)
$$ 
for $z \in \mathbb U$, and 
in particular, 
$\Vert \nu_u \Vert_{\infty} = \Vert \mu_{\tilde u} \Vert_{\infty}$. 
We will further prove that there exists a constant $C > 0$ such that 
$$
\Vert \nu_u \Vert_{\infty} + \Vert \nu_u \Vert_{c} \leq C (\Vert \mu_{\tilde u} \Vert_{\infty} + \Vert \mu_{\tilde u} \Vert_{c}),
$$
and moreover $\nu_u \in \mathcal M_0(\mathbb D)$ by $\mu_{\tilde u} \in \mathcal M^{\#}_0(\mathbb U)$. These properties show that the projection from $\mathcal M^{\#}_0(\mathbb U)$ to $\mathcal M_0(\mathbb D)$ defined by $\mu_{\tilde u} \mapsto \nu_u$ is a bounded linear map, and thus it is real-analytic. 

For any $h \in (0,1]$ and $\theta_0 \in [0,2\pi)$, we take a sector $S_{h,\theta_0}$ defined by \eqref{sector}.
If $h \in (\frac{1}{2},1]$, then
by $e^{2\pi y} - 1 \geq 2\pi y$ for $y > 0$, 
we have
\begin{equation}\label{case1}
\begin{split}
&\quad\frac{1}{h}\iint_{S_{h,\theta_0}} \frac{|\nu_u (\zeta)|^2}{1 - |\zeta|^2} d\xi d\eta \leq \frac{1}{h}\iint_{|\zeta| < \frac{1}{2}} \frac{|\nu_u  (\zeta)|^2}{1 - |\zeta|^2} d\xi d\eta + \frac{1}{h}\iint_{\frac{1}{2} \leq |\zeta|<1} \frac{|\nu_u (\zeta)|^2}{1 - |\zeta|^2} d\xi d\eta \\
& \leq \frac{2\pi}{h} \int_0^{1/2} \frac{rdr}{1-r^2}\,\Vert \nu_{u} \Vert_{\infty}^{2} + \frac{4\pi^2}{h}\int_{0}^{\log 2/(2\pi)}\int_{-1/2}^{1/2} \frac{|\nu_u (e^{2\pi iz})|^2}{1 - |e^{2\pi iz}|^2}|e^{2\pi iz}|^2 dxdy\\
& \leq \frac{\pi}{3h} \Vert \mu_{\tilde u} \Vert_{\infty}^{2} + \frac{\pi}{h} \int_{0}^{1}\int_{-1/2}^{1/2} \frac{|\mu_{\tilde u}(z)|^2}{y} dxdy \lesssim \Vert \mu_{\tilde u} \Vert_{\infty}^{2}+\Vert \mu_{\tilde u} \Vert_{c}^{2}.
\end{split}
\end{equation}
Similarly, if $h \in (0, \frac{1}{2}]$, we have
\begin{equation}\label{firstline}
\begin{split}
\quad\frac{1}{h}\iint_{S_{h,\theta_0}} \frac{|\nu_u (\zeta)|^2}{1 - |\zeta|^2} d\xi d\eta 
& \leq \frac{\pi}{h} \int_{0}^{\frac{1}{2\pi}\log \frac{1}{1-h}}\int_{\theta_0/(2\pi)- h/2}^{\theta_0/(2\pi)+h/2} \frac{|\mu_{\tilde u}(z)|^2}{y} dxdy\\
& \leq \frac{\pi}{h} \int_{0}^{h}\int_{\theta_0/(2\pi)- h/2}^{\theta_0/(2\pi)+h/2} \frac{|\mu_{\tilde u}(z)|^2}{y} dxdy 
\lesssim \Vert \mu_{\tilde u}\Vert_c^2. 
\end{split}
\end{equation}

Therefore, we have $\Vert \nu_u  \Vert_{\infty} + \Vert \nu_u  \Vert_{c} \leq C (\Vert \mu_{\tilde u} \Vert_{\infty} + \Vert \mu_{\tilde u} \Vert_{c})$ for some constant $C > 0$. On the other hand, we also see that the left side of the first line of \eqref{firstline} tends to $0$ uniformly as $h \to 0$ since the right side of the first line of \eqref{firstline} tends to $0$ uniformly as $h \to 0$ by 
$\mu_{\tilde u} \in \mathcal M_0(\mathbb U)$. Then, we have  $\nu_u \in \mathcal M_0(\mathbb D)$.  
 
We can conclude that $\widetilde \Lambda$ is the composition of the following three maps:
\begin{align*}
{\rm VMO}_{\mathbb R}(\mathbb S) \ni u &\mapsto \tilde u \in {\rm VMO}_{\mathbb R}^{\#}(\mathbb R),\\ 
{\rm VMO}_{\mathbb R}^{\#}(\mathbb R) \ni \tilde u &\mapsto  \mu_{\tilde u} \in \mathcal M^{\#}_0(\mathbb U),\ {\rm and}\\
\mathcal M_0^{\#}(\mathbb U) \ni \mu_{\tilde u} &\mapsto \nu_u  \in \mathcal M_0(\mathbb D),
\end{align*}
each of which is real-analytic. Thus, the composition $\widetilde \Lambda$ is real-analytic. This completes the proof of Theorem \ref{realana}.
\end{proof}

\end{document}